\RequirePackage{fix-cm}
\documentclass[smallextended]{svjour3}       
\smartqed  

\usepackage{booktabs}
\usepackage{adjustbox}
\usepackage{graphicx}
\usepackage{amsmath}
\usepackage{amssymb}
\usepackage{multirow}
\usepackage{url}
\usepackage[tight,footnotesize]{subfigure}
\usepackage{hyperref}
\usepackage[switch, modulo]{lineno}
\usepackage{setspace}
\usepackage{amsmath}

\DeclareMathOperator*{\argmin}{arg\,min}

\usepackage{algcompatible}
\usepackage{algorithm}
\usepackage[noend]{algpseudocode}

\makeatletter
\def\BState{\State\hskip-\ALG@thistlm}
\makeatother

\begin{document}

\title{Black-box optimization on hyper-rectangle using Recursive Modified Pattern Search and application to ROC-based Classification Problem}

\titlerunning{Recursive Modified Pattern Search}        

\author{Priyam Das}


\institute{Priyam Das \at
              Department of Biostatistics, Virginia Commonwealth University, Richmond, VA, USA \\
               Tel.: +1 919-308-5892\\
              \email{dasp4@vcu.edu}}

\date{Received: date / Accepted: date}

\maketitle

\begin{abstract}
In statistics, it is common to encounter multi-modal and non-smooth likelihood (or objective function) maximization problems, where the parameters have known upper and lower bounds. This paper proposes a novel derivative-free global optimization technique that can be used to solve those problems even when the objective function is not known explicitly or its derivatives are difficult or expensive to obtain. The technique is based on the pattern search algorithm, which has been shown to be effective for black-box optimization problems. The proposed algorithm works by iteratively generating new solutions from the current solution. The new solutions are generated by making movements along the coordinate axes of the constrained sample space. Before making a jump from the current solution to a new solution, the objective function is evaluated at several neighborhood points around the current solution. The best solution point is then chosen based on the objective function values at those points. Parallel threading can be used to make the algorithm more scalable. The performance of the proposed method is evaluated by optimizing up to 5000-dimensional multi-modal benchmark functions. The proposed algorithm is shown to be up to 40 and 368 times faster than genetic algorithm (GA) and simulated annealing (SA), respectively. The proposed method is also used to estimate the optimal biomarker combination from Alzheimer's disease data by maximizing the empirical estimates of the area under the receiver operating characteristic curve (AUC), outperforming the contextual popular alternative, known as step-down algorithm.
\end{abstract}
\newpage
\section{Introduction}
\label{sec1}
In science and engineering, a black box is a system whose inner workings are not known or cannot be easily determined. It can be characterized solely by its input and output data. In operations research and machine learning, a black-box function is a function whose internal structure is unknown and cannot be easily inferred from its input and output data. These functions are often characterized by presence of multiple local extremums, non-analytical derivatives, and discontinuities. While optimizing an objective function with a complex mathematical structure and a large number of variables, it is often common practice to treat the function as a black box. This is because the uncertainty about the continuity and smoothness of the function can limit the applicability of convex optimization techniques. Additionally, the multi-modal nature of the function can make it unsuitable for convex optimization techniques.

`Gradient descent (GD)' (\cite{Marquardt1963}) is a well-established method for convex optimization. Other convex optimization techniques include the `Trust Region Reflective (TRF)' algorithm (\cite{Byrd2000}), `Interior-point (IP)' algorithm (\cite{Potra2000})  and `Sequential Quadratic Programming (SQP)' algorithm (\cite{Boggs1996}). Most of these algorithms (e.g., GD, TRF, SQP) use derivatives to find the best direction of movement in the sample space while optimizing an objective function. While minimizing an objective function using convex optimization techniques, the focus remains on finding a local solution. Although these algorithms are efficient for minimizing convex problems, they can get stuck at local solutions while minimizing non-convex functions. For low-dimensional non-convex minimization problems, it is common to use convex optimization techniques starting from multiple starting points and select the best solution from the set of obtained solutions. However, this strategy can be computationally expensive for high-dimensional problems. Some optimization algorithms (e.g., \cite{Goodner2012}) prioritize minimizing the number of function evaluations required to find a reasonable local minimum. The proposed strategy may accelerate the algorithm's execution time, but it may not be suitable for applications that require a high-quality global solution rather than a fast execution time.

Over the past few decades, several heuristic algorithms have been proposed for optimizing black-box functions. These algorithms typically incorporate principles of exploration and avoidance of derivative-based movements. One of the earliest and most popular black-box optimization techniques is the Genetic algorithm (GA) \cite{Fraser1957}, which was later popularized by \cite{Bethke1980,Goldberg1989}. Another popular heuristic algorithm is simulated annealing (SA) \cite{Kirkpatrick1983}) which is widely used in the field of engineering (\cite{Granville1994}). Particle swarm optimization (PSO) is another commonly used heuristic algorithm which was proposed by \cite{Kennedy1995}. These heuristic algorithms typically incorporate a balance between exploration and exploitation. Exploration refers to the search for new and potentially better solutions, while exploitation refers to the refinement of promising solutions. Once it gets stuck to any local minimum, further exploration efforts are typically made to find better solution following their corresponding heuristic principles. In general, derivative-based movements are avoided in these algorithms, as the computation of derivatives of blackbox functions can be expensive, if not intractable. Thus the sense of exploration and avoiding usage of derivatives are two notable aspects conceived by any typical blackbox techniques unlike convex optimization approaches.  While these heuristic algorithms have been shown to be effective for a wide range of blackbox optimization problems, they can still face challenges in high-dimensional spaces., e.g., GA might not scale well with complexity as in higher dimensional optimization problems since there is often an exponential increase in search space size (\cite{Geris2012}).

A common and effective approach to global optimization without using derivatives is to search along the coordinates. A sequence of coordinate search-based global optimization techniques (e.g., \cite{Huyer1999}) have been developed over time, and one of the simplest and most effective is the coordinate search algorithm proposed by Fermi (\cite{Fermi1952}) in 1952. The algorithm works by first setting a step size at the beginning of each iteration. Then, the step size is added and subtracted from each coordinate of the current solution, one at a time. This results in $2n$ new points being explored in the neighborhood of the current solution in a $n$-dimensional parameter space. The objective function value is evaluated at each of these points, and the point with the minimum objective function value is taken as the updated solution. If no improvement is observed after an iteration, the step size is halved and the search is repeated. This allows for a finer search of the parameter space. \cite{Torczon1997} extended and generalized the Fermi's principle for a number of black-box optimization problems to propose the Generalized pattern search (GPS). Pattern Search (PS) algorithms are a class of optimization algorithms that search for the minimum of an objective function by evaluating the function value at a few points in the neighborhood of the current solution. These points are obtained by making coordinate-wise movements, or changes to one variable at a time. The best point in the neighborhood is then selected as the new current solution. The performance of PS algorithms depends on the way the sizes of the coordinate-wise movements are chosen. Later several articles evolved focusing on optimization using Direct Search (DS) which can be considered as a variant of PS (\cite{Audet2006,Custodio2015,Martinez2013,Kolda2003,Lewis1999}). A few other derivative-free optimization methods have been proposed in \cite{Audet2014,Conn2009,Jones1998,Digabel2011,Martelli2014,Audet2008,Audet22008}. 

A plethora of black-box optimization algorithms have been developed in recent years. However, only a small number of these algorithms can handle constrained black-box problems. This is because unconstrained black-box optimization algorithms may evaluate the objective function at points that are infeasible, i.e., outside the constrained domain. This can lead to wasted computational resources and sub-optimal solutions. In statistics, it is common for the parameters to be bounded by known upper and lower bounds. Even if the parameters are not theoretically bounded, it can be helpful to place reasonable bounds on them based on prior knowledge. This can help to reduce the search space and improve the efficiency of the optimization algorithm. An algorithm specifically tailored to optimize any objective function over any desired bound on the coordinates can be beneficial in two ways: (1) it can reduce the search space by incorporating domain knowledge and prior data, and (2) it can conduct a more thorough search over the desired region by considering the coordinate-wise lower and upper bounds. 

\begin{figure}[ht]
	\centering
	\includegraphics[width=0.8\textwidth]{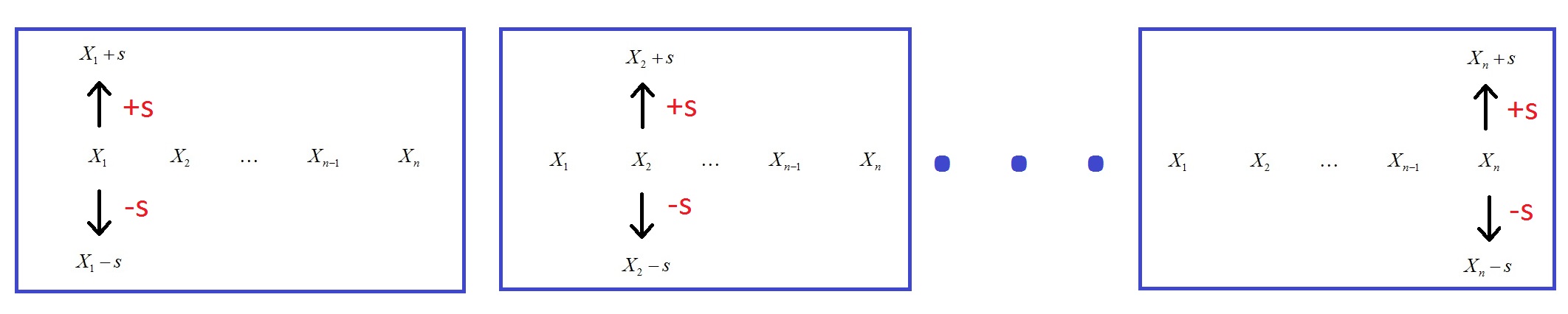}
	\caption{Fermi's principle : Possible $2n$ movements starting from initial point $(x_1,\ldots,x_n)$ inside an iteration with fixed step-size $s$, while optimizing any $n$-dimensional objective function.}
	\label{fermi}
\end{figure}

The proposed algorithm, similar to the Fermi's principle (\cite{Fermi1952}, see Figure \ref{fermi}), evaluates the objective function at $2n$ neighboring points at each iteration. These points are obtained by making $2n$ coordinate-wise movements with derived step sizes, where n is the dimension of the parameter space. Unlike the Fermi's principle and GPS, RMPS is specifically designed for optimizing blackbox functions over hyper-rectangular domains. Appendix A details the other differences between RMPS and existing PS techniques. As is typical for blackbox optimization algorithms, RMPS includes a strategy for jumping out of local minima in search of better solutions. This strategy, called the `restart' strategy, is discussed as follows. The step size values for coordinate-wise movements in RMPS can be used as an indicator of whether a local minimum has been found. (for details see Section \ref{sec_theory}). When the magnitude of all movement steps falls below a user-defined threshold, the step size is reset to a higher value in order to explore a more distant neighborhood for a better solution. We refer to this strategy as `restart', which, to the best of our knowledge, has not been previously considered in the field of PS algorithms. The algorithm terminates when the difference between two consecutive `restart' solutions is less than a user-defined threshold. One of the algorithm's key advantages is that it can evaluate objective functions in parallel in $2n$ directions within each iteration. This is because the step size for an iteration is fixed, and the movements and objective function value evaluation steps in $2n$ directions are independent of each other.

The rest of the paper is organized as follows. Section \ref{sec_algo} provides the detailed description of RMPS. In Section \ref{sec_theory}, we discuss the theoretical properties of RMPS. Section \ref{sec_app} presents a comparative performance analysis of RMPS and several other existing methods, based on their optimization of several moderate and high-dimensional challenging non-convex benchmark functions In Section \ref{sec_HUM}, the RMPS is applied to find the optimal linear combination of the biomarkers for differentiating the Alzheimer's patients of different severity levels. A brief concluding remarks regarding RMPS is discussed in Section \ref{sec_discussion}.

\section{Recursive Modified Pattern Search (RMPS)}
\label{sec_algo}
Consider an objective function $Y = f(\boldsymbol{x})$ where $\mathbf{x} = (x_1,\cdots,x_n)$ is the parameter of dimension $n$. Proposed optimization problem is as follows:
\begin{align*}
	minimize \;:& \; f(\mathbf{x}) \nonumber \\
	subject \; to \; :& \; \mathbf{x} \in S \subset \mathrm{R}^n 
\end{align*}
where $S=\prod_{j=1}^n I_j$, $I_j=[a_j,b_j]$ are closed and bounded intervals on $\mathrm{R}$ for $i=1,\cdots,n$. Now consider the bijection map
\begin{align*}
	g:S \mapsto [0,1]^n
\end{align*}
where $g(\mathbf{z}) = (g_1(\mathbf{z}), \ldots, g_n(\mathbf{z})) \in [0,1]^n$ is such that $g_i(\mathbf{z}) = \frac{z_i-a_i}{b_i-a_i}$. So, without loss of generality, the domain of $\mathbf{x}$ can be taken to be $S = [0,1]^n$. Hence, the problem of minimizing a function over a hyper-rectangular domain is the same as the problem of minimizing a function over a unit cube of the same dimension. In the rest of this paper, for the sake of convenience, the domain of the parameter space is considered to be a unit cube of dimension $n$ .

\paragraph{\textit{Runs} :}
RMPS consists of several \textit{runs}. Each \textit{run} is given by a sequence of iterations and a \textit{run} is terminated based on some convergence criteria (details follows). After iteration step, typically a better solution point might be found. Although, it is also possible to fail to obtain a better solution after an iteration. Upon termination of a \textit{run}, the solution obtained at the final iteration of that \textit{run} is returned. The initial \textit{run} is initialized with the user-provided starting point. Subsequent \textit{run}s start with the solution returned by the previous \textit{run}. For example, the fourth \textit{run} would start with the solution returned by the third \textit{run}. Therefore, the user must specify the starting point for the initial \textit{run} only (see Figure \ref{flowchart}). Starting from the solution returned by the previous \textit{run}, the following \textit{run} tries to improve the solution by making coordinate-wise changes to the current solution. The step sizes for these changes are typically either the same or smaller than the step sizes used in the previous iteration. The sequence of step sizes gradually decreases towards zero as more iterations are executed within any particular \textit{run}. Thus the solution quality improves or remains the same after each \textit{run}. If two consecutive \textit{runs} return the same solution (up to some approximation set by the user), the algorithm terminates and returns the solution returned by the last executed \textit{run}, which is the final solution. 

In order to explain the role of the parameters considered in RMPS, based on the location where their values are provided/updated, we divide them into 3 broad categories :
\begin{itemize}
    \item Parameters set at the beginning, before executing the first \textit{run}.
    \item Parameters set at the beginning of \textit{run}, before executing first iteration.
    \item Parameters set within each iteration.
\end{itemize}

\begin{figure}[ht]
	\centering
	\includegraphics[width=0.8\textwidth]{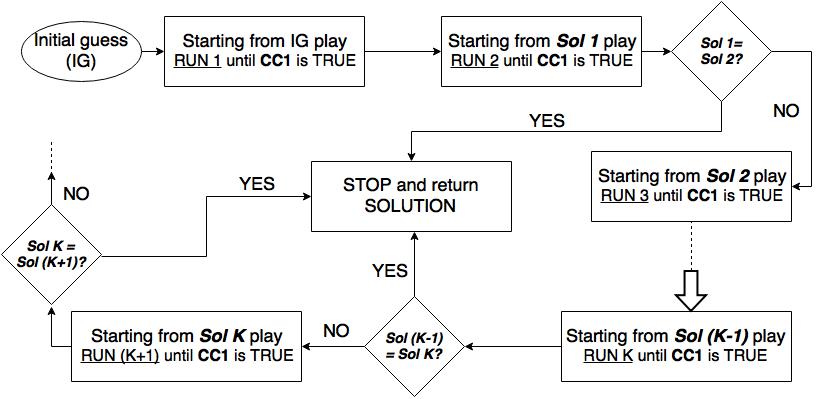}
	\caption{Flowchart of RMPS algorithm. Here CC1 denotes the condition ($|F-min(F,FF)| < tol\_fun$ and $s<\phi$), see Algorithm pseudo-code part in Section \ref{sec_algo} for details.}
	\label{flowchart}
\end{figure}

\paragraph{Parameters set at the beginning, before executing the first \textit{run} :}
The following parameters are set at the beginning of the algorithm, before the first \textit{run} is executed.
\begin{description}
\item[$\bullet$ $\textit{max\_runs}$ :]$\textit{max\_runs}$ denotes the maximum number of \textit{runs} allowed in the algorithm.

	\item[$\bullet$ $\textit{max\_iter}$ :]$\textit{max\_iter}$ denotes the maximum number of iterations allowed inside a \textit{run}. 
    
	\item[$\bullet$ $\textit{tol\_fun}$ :] $\textit{tol\_fun}$ determines the minimum amount of improvement required after an iteration so that the \textit{global step size} (described later) is kept unchanged for next iteration. 
    
	\item[$\bullet$ $\textit{tol\_fun\_2}$ :] Second \textit{run} onward, whenever a \textit{run} terminates, it is checked whether the solution returned by the current \textit{run} is the same or different with the solution returned by the previous \textit{run}, subjected to some approximation. If $\hat{\mathbf{x}}^{(R)}$ and $\hat{\mathbf{x}}^{(R-1)}$ denote the solution obtained by $R$-th and $(R-1)$-th \textit{run}s, then if $||\hat{\mathbf{x}}^{(R)} - \hat{\mathbf{x}}^{(R-1)}||<\textit{tol\_fun\_2}$ holds true, then no further \textit{run} is executed and $\hat{\mathbf{x}}^{(R)}$ is returned as the final solution. Otherwise, next \textit{run} is executed.
\end{description}

\paragraph{Parameters set at the beginning of \textit{run}, before executing first iteration :}
In RMPS each \textit{run} follows similar mechanism, except the values of the tuning parameters can be reset after each \textit{run}. There are three tuning parameters which are initial \textit{global step size} $s_{initial}$, \textit{step decay rate} $\rho$, \textit{step size threshold} $\phi$ respectively. Note that it is not necessary to change their values at each \textit{run}. 
A brief description of these parameters are provided as follows:  
\begin{description}
\item [$\bullet$ initial \textit{global step size} ($s_{initial}$) :] $s_{initial}$ denotes the highest possible size of coordinate-wise jumps. 
	\item[$\bullet$ \textit{step decay rate} ($\rho$) :] $\rho$ determines the rate of change of \textit{global step size} at the end of each iteration. It's value is fixed at the beginning of each \textit{run}. Whenever step-sizes are required to be diminished, the step-size is divided by this quantity, $\rho (>1)$. 

\item[$\bullet$ \textit{step size threshold} ($\phi$) :] This is the lower bound of the all step-sizes (\textit{global step size} and the \textit{local step sizes}; discussed in the next subsection). Once any step size reaches below $\phi$, it is considered that the step-size is sufficiently close to zero.
	
    
\end{description}
\paragraph{Parameters set within each iteration :}
Within each iteration there is a parameter called \textit{global step size} (denoted by $s^{(j)}$ for $j$-th iteration) along with $2n$ local parameters called \textit{local step sizes}, denoted by $\{s_i^+\}_{i=1}^n$ and $\{s_i^-\}_{i=1}^n$. Note that their values are not set by the user. Their values are updated automatically based on the improvement of the objective function value over the iterations.
Further, Appendix B summarizes a brief description on the relationship between the performance of RMPS algorithm and the values of the tuning parameters. 

\paragraph{RMPS mechanism within a \textit{run} :}
Inside a \textit{run}, the first iteration is performed setting the \textit{global step size} $s^{(1)}= s_{initial}$. Inside an iteration, the value of the \textit{global step size} is kept unchanged throughout all the operations. At the beginning of any iteration, the \textit{local step size} $\{s_i^+\}_{i=1}^n$ and $\{s_i^-\}_{i=1}^n$ are set to be equal to the \textit{global step size} of the corresponding iteration. For example, at the beginning of $j$th iteration, we set $s_i^+ = s_i^- = s^{(j)}$ for $i=1,\cdots,n$. 

Suppose the current value of $\mathbf{x}$ at the $j$-th iteration is $\mathbf{x}^{(j)} = (x_1^{(j)}, \ldots, x_n^{(j)})$. If incrementing $x_i^{(j)}$ by  $s_i^+$ generates a point outside the domain, i.e., $(x_i^{(j)} + s_i^+)>1$, then $s_i^+$ is updated to the value $\frac{s^{(j)}}{\rho^f}$ where $f$ is the smallest possible integer such that $x_i^{(j)} + \frac{s^{(j)}}{\rho^f} <1$. Similarly if decrementing $x_i^{(j)}$ by  $s_i^-$ generates a point outside the domain (i.e., $(x_i^{(j)} - s_i^-)<0$), then $s_i^-$ is updated to the value $\frac{s^{(j)}}{\rho^f}$ where $f$ is the smallest possible integer such that $x_i^{(j)} - \frac{s^{(j)}}{\rho^f} >0$. While choosing $f$ for any given coordinate and given direction, it is taken into account that the updated \textit{local step size} is greater than \textit{step size threshold} $\phi$. In case such a $f$ is not feasible (for example, if $1-x_i^{(j)}<\phi$, then no such $f$ exists such that $x_i^{(j)} + \frac{s^{(j)}}{\rho^f} <1$ and $\frac{s^{(j)}}{\rho^f}>\phi$ ), that particular coordinate is not updated in the corresponding direction and is kept unchanged.

After assigning the values of the \textit{local step sizes} in above-mentioned fashion, $2n$ new points are obtained in the neighborhood of the current solution $\mathbf{x}^{(j)}$. Objective functions are evaluated at those points. Out of these $2n+1$ points including $\mathbf{x}^{(j)}$, the point corresponding to lowest objective function value is considered as the updated solution $\mathbf{x}^{(j+1)}$. In case $|f(\mathbf{x}^{(j+1)})-f(\mathbf{x}^{(j+1)})|<\textit{tol\_fun}$ holds, then \textit{global step size} is decremented via being divided by $\rho$, i.e.,  $ s^{(j+1)}$ is set equal to $ s^{(j)}/\rho$. Otherwise, \textit{global step size} is kept unchanged. The motive behind this step is that if the objective function is not improving well for the current value of \textit{global step size}, it is decremented to employ a finer search. Once the value of \textit{global step size} becomes less than \textit{step size threshold} $\phi$, it implies the \textit{global step size} is sufficiently close to zero. Which signifies a local solution is reached (up to some approximation) under certain sets of regularity conditions, as described later in Section \ref{sec_theory}. Therefore, the \textit{run} is terminated upon \textit{global step size} becoming less than $\phi$.

Within an iteration, there is only one \textit{global step size} along with $2n$ \textit{local step sizes} which are initialized at the same value as that of \textit{global step size}. At the end of the iteration, each of them ends up being less than or equal to the \textit{global step size} of that iteration. It should be noted that in a \textit{run}, the \textit{global step size} might decrease or remain same after each iteration. On the other hand, the values of the \textit{local step sizes} do not depend on their values in the previous iteration as it's initial value is set to be equal to \textit{global step size} at the beginning of each iteration. A \textit{run} terminates when \textit{global step size} becomes smaller than $\phi$.

\paragraph{RMPS restart strategy :}
For considering sufficiently small \textit{step size threshold} value, under regularity conditions it can be shown that RMPS returns the global minimum (see Section \ref{sec_theory}). Also the \textit{run} terminates when \textit{global step size} becomes smaller than \textit{step size threshold}. However for non-convex functions, it is hard to verify if that is a global minimum or not. To increase the likelihood of reaching the global minimum, at the end of a \textit{run}, another \textit{run} is initiated setting the \textit{global step size} to initial \textit{global step size} $s_{initial}$. Thus at the beginning of the next \textit{run}, the subsequent set of candidate points are chosen from distant neighborhood and this heuristic approach increases the likelihood of jumping out of the current local solution. Note that, in case it is known that the objective function is convex, there is no need to `restart' since each \textit{run} is designed in such a way that it returns a local solution (see Section \ref{sec_theory}) which is the global solution in case of convex objective function. therefore, execution of only one \textit{run} is sufficient for minimizing a convex function, . 

Regarding this heuristic `restart' strategy, a concern might be raised whether restarting the next \textit{run} from the solution returned by the previous \textit{run} is more justified compared the possible alternative of restarting from different point(s) randomly generated from the parameter space. Firstly, under certain regularity conditions, the solution obtained at the end of a \textit{run} is a local solution and has the potential of being the true global solution as well, which is not necessarily true for any other randomly generated point. Now, another option is to restart from multiple randomly generated points and execute RMPS for each one of them. For smaller number of \textit{run}s, this strategy might be scalable. But as the number of \textit{run}s increase, the set of possible restart points will increase exponentially. Therefore multiple starting point strategy might prove to be computationally infeasible. Further, studies in Section \ref{sec_app} reveal that RMPS yields almost equally good solution starting from different randomly generated points while minimizing high-dimensional benchmark functions. Thus implementing multiple-point `restart' strategy for RMPS seems redundant. 

\paragraph{Algorithm pseudocode :}
Pseudocode for RMPS is given in \textbf{Algorithm 1}. Let $\hat{\mathbf{x}}^{(R)}$ denote the solution obtained at the end of $R$-th \textit{run}, and $\mathbf{x}^{(j)}$ denote the solution obtained after the $j$-th iteration within a particular \textit{run}. Initial guess of the solution $\mathbf{x}^{(1)}$ is provided by the user.

\begin{algorithm}
\caption{Pseudocode for RMPS}\label{euclid}
{\fontsize{7}{7}
\begin{algorithmic}[]
\State $R \gets 1$
\BState \emph{top}:
\State $j \gets 1$
\State $s^{(0)},s^{(1)}  \gets s_{initial}$
\If {$(R = 1)$}
\State $\mathbf{x}^{(0)} \gets \text{Initial guess}$
\Else 
\State $\mathbf{x}^{(0)} \gets \hat{\mathbf{x}}^{(R-1)}$
\EndIf
\While {($j \leq max\_iter$ and $s^{(j)} > \phi$)}
\State $F \gets f(\mathbf{x}^{(j-1)})$
\State $s \gets s^{(j-1)}$
\For {$h = 1:2d$}
\State $i \gets [\frac{(h+1)}{2} ]$ ($[\cdot]$ denotes largest smaller integer function)
\State $ s_h \gets (-1)^hs^{(j)}$
\State $\mathbf{x}_h \gets \mathbf{x}^{(j-1)}$
\State $\mathbf{x}_h(i) \gets \mathbf{x}_h(i)+s_h$
\If {($\mathbf{x}_h(i)>1$)}
\State $f = \big[\log_\rho \frac{s_h}{(1-\mathbf{x}^{(j)}(i))}\big]+1$
\State $s_h \gets s_h/\rho^f$
\EndIf
\If {($\mathbf{x}_h(i)<0$)}
\State $f = \big[\log_\rho \frac{-s_h}{\mathbf{x}^{(j)}(i)}\big]+1$
\State $s_h \gets s_h/\rho^f$
\EndIf
\If {($s_h> \phi$)}
\State $f_h \gets f(\mathbf{x}_h)$
\Else 
\State $f_h \gets F$
\EndIf
\EndFor
\State $\mathbf{B} \gets \argmin_{\mathbf{x}_h} f_h$
\State $FF \gets min(\{f_h\}_{h=1}^{2d})$ 
\If {($FF < F$)}
\State $\mathbf{x}^{(j)} \gets \mathbf{B}$
\Else
\State $\mathbf{x}^{(j)} \gets \mathbf{x}^{(j-1)}$
\EndIf
\If {($j > 1$)}
\If {($|F-min(F,FF)| < tol\_fun$ and $s>\phi$)}
\State $s \gets \frac{s}{\rho}$
\EndIf
\EndIf
\State $s^{(j)} \gets s$
\State $j \gets j+1$
\EndWhile
\State $\hat{\mathbf{x}}^{(R)} \gets \mathbf{x}^{(j)}$, 
\If {$||\hat{\mathbf{x}}^{(R)} - \hat{\mathbf{x}}^{(R-1)}|| < tol\_fun\_2$ }
\State \Return $\hat{\mathbf{x}} = \hat{\mathbf{x}}^{(R)}$ as \textbf{final solution}
\State \textbf{exit}
\Else
\State $R \gets R+1$
\State \textbf{goto} \emph{top}.
\EndIf 
\end{algorithmic}}
\end{algorithm}

\section{Theoretical properties}
\label{sec_theory}
While optimizing any black-box function, finding global solution cannot be possibly guaranteed by using any algorithm. However, under restrictive conditions like convexity, it is desirable for any optimization algorithm to reach the global minimum. In this section we show that under certain regularity conditions, RMPS algorithm reaches a global minimum of the objective function. The following theorem shows that under the assumed regularity conditions, the stopping criteria of RMPS ensures that a global solution (up to some approximation error) is returned.
\begin{theorem}
\label{theorem}
	Suppose $f : \mathrm{R}^n \mapsto \mathbb{R}$ is locally convex and differentiable on a compact set $C \subset \mathrm{R}^n$. Consider a point $\mathbf{u} = (u_1,\ldots, u_n) \in C$. Consider a sequence $\delta_k = \frac{s}{\rho^k}$ for $k\in \mathrm{N}$ and $s>0, \rho>1$. Define $\mathbf{u}_k^{(i+)} = (u_1,\ldots, u_{i-1},u_i+\delta_k,u_{i+1}, \ldots, u_n)$ and $\mathbf{u}_k^{(i-)} = (u_1,\ldots, u_{i-1},u_i-\delta_k,u_{i+1}, \ldots, u_n)$ for $i=1,\cdots,n$. If for all $k \in \mathrm{N}$, $f(\mathbf{u}) \leq f(\mathbf{u}_k^{(i+)})$ and $f(\mathbf{u}) \leq f(\mathbf{u}_k^{(i-)})$ for all $i = 1,\cdots,n$, the global minimum of $f$ on $C$ occurs at $\mathbf{u}$.
\end{theorem}

\begin{proof}[Proof of Theorem \ref{theorem}]
	Take an open neighborhood $\mathbf{U} \subset C$ w.r.t. $l_{\infty}$-norm containing $\mathbf{u}$ at the center. So, there exists $r>0$ such that $\mathbf{U} = \prod_{i=1}^nU_i$ where $U_i = (u_i-r,u_i+r)$ for $i=1,\ldots,n$. For some $i \in \{1,\ldots,n\}$, define $g_i :  U_i \mapsto \mathrm{R}$ such that $g_i(z) = f(u_1,\ldots,u_{i-1}, z,u_{i+1}, \ldots u_n)$. Since $f$ is convex on $\mathbf{U}$, it can be easily shown that $g_i$ is also convex on $U_i$. We claim that $g_i(u_i) \leq g_i(z)$ for all $z \in U_i$.
	
	Suppose there exist a point $u_i^* \in U_i$ such that $g_i(u_i^*) < g_i(u_i)$. Take $d = |u_i^* - u_i|$. 
 (see Figure S1 and S2 in the appendix).
 Clearly $0 < d < r$. Without loss of generality, assume $u_i^* > u_i$. Hence $u_i^* = u_i + d$. Since $\delta_k$ is a strictly decreasing sequence going to $0$, there exists a $N$ such that for all $k \geq N$, $\delta_k < d$. Now we have $u_i < u_i + \delta_N < u_i + d$. Now there exists a $\lambda \in (0,1)$ such that $u_i+\delta_N = \lambda u_i + (1-\lambda)(u_i+d)$. From convexity of $g_i$, we have
	\begin{align*}
		f(u_i+\delta_N)  &= f(\lambda u_i + (1-\lambda)(u_i+d))   \\
		& \leq \lambda f(u_i) + (1-\lambda)f(u_i+d) \\
		& = \lambda f(u_i) + (1-\lambda)f(u_i^*) \\
		 &= f(u_i)  - (1 - \lambda)(f(u_i) - f(u_i^*)),
	\end{align*}

	Since, $ f(u_i^*) < f(u_i)$, it implies $f(u_i+\delta_N)<f(u_i)$. But, we know $f(u_i) \leq f(u_i+\delta_N)$. Hence it is a contradiction. 
	
	Since partial derivatives of $f$ exist at $\mathbf{x} = \mathbf{u}$, $g_i$ is differentiable at $z =u_i$. Since $u_i$ is a local minima of $g_i$ in $U_i$, we have $g_i^{\prime}(u_i) = 0$. So we have $\frac{\partial}{\partial x_i}f(\mathbf{x})|_{\mathbf{x} = \mathbf{u}} =0$. By similar argument it can be shown that $\frac{\partial}{\partial x_j}f(\mathbf{x})|_{\mathbf{x} = \mathbf{u}} =0$ for all $j = 1,\ldots, n$. Since $f$ is convex and $\nabla f(\mathbf{u}) = 0$, $\mathbf{u}$ is a local minima in $\mathbf{U}$. Now, since $\mathbf{U} \subset C$, $\mathbf{u}$ is also a local minima of $C$. But $f$ is convex on $C$. Since any local minimum of a convex function is necessarily global minimum, the global minimum of $f$ occurs at $\mathbf{u}$. 
\end{proof}
Suppose the objective function is convex and all the partial derivatives exist at the obtained solution $\mathbf{u}\in C$ which is an interior point. The proposed algorithm terminates when two consecutive \textit{runs} yield the same solution. It implies in the last \textit{run}, the objective function values at all the sites obtained by making jumps of sizes $\delta_k = \frac{s_{initial}}{\rho^k}$ (until $\delta_k$ gets smaller than \textit{step size threshold}) around $\mathbf{u}$, i.e. $f(\mathbf{u}_k^{(i+)})$ and $f(\mathbf{u}_k^{(i-)})$, are greater than or equal to $f(\mathbf{u})$ for $i=1,\ldots,n$. So, taking \textit{step size threshold} sufficiently small in RMPS ensures that the returned solution is the global minimum under the regularity conditions. From Theorem \ref{theorem} it is concluded that if the objective function is convex and differentiable then taking \textit{step size threshold} sufficiently small yields the global minimum. It should be also noted that if the function is convex and differentiable and it takes minimum value at some interior point, evaluation of only the first \textit{run} is sufficient to obtain the solution.

\section{Comparative study on Benchmark functions}
\label{sec_app}
In this section, the comparative performance study of RMPS, Simulated Annealing (SA) and  Genetic Algorithm (GA) is considered. Both SA and GA are available in Matlab R2014a (The Mathworks) via the Optimization Toolbox functions \textit{simulannealbnd} (for SA), and \textit{ga} (for GA) respectively. In the comparative study, the maximum number of allowed iterations and evaluations of objective function is set to be infinity for \textit{simulannealbnd} function. In case of \textit{ga}, the default parameter values are used. RMPS is implemented in Matlab 2014a and the code is made available  \href{https://github.com/priyamdas2/RMPS}{\textbf{here}}. The values of the tuning parameters are taken to be as follows: $\textit{max\_runs}=1000, \textit{max\_iter}=5000,  \textit{tol\_fun} = 10^{-15}, \textit{tol\_fun\_2} = 10^{-6}, \rho_1=2$ (\textit{step decay rate} for the first \textit{run}), $\rho_2=1.05$ (\textit{step decay rate} for second \textit{run} onwards),$\phi=10^{-6}$. We consider 45 benchmark functions (\cite{Jamil2013}) and  each test function is minimized starting from 10 randomly generated points (under 10 random number generating seeds in MATLAB) within the considered domain of solution space using these three above-mentioned methods. The domains of the search regions can be found in Table S1 of the appendix. All simulation studies in this paper are performed in a machine with 64-Bit Windows 8.1, Intel i7 3.60GHz processors and 32GB RAM. In Table \ref{comparison_table_1}, the obtained minimum values for each occasion is noted down for all considered algorithms along with average computation time. It is noted that RMPS and GA perform more or less better than SA. It should be also noted that RMPS, in general, yields reasonable solution in lesser time than GA and SA in most of the cases. Using RMPS upto 9 and 15 folds improvement in computation times are obtained over GA ans SA respectively. Although GA yielded better solutions in some of the cases at the cost of more computation time, it should be noted that in the case of RMPS, taking the value of \textit{step size threshold} ($\phi$) smaller than its default value (i.e., $10^{-6}$), better solution might be obtained with more computation time. Because  smaller value of \textit{step size threshold} would employ finer search (possibly) resulting in more accurate solution.

As mentioned in Section \ref{sec_theory}, to minimize any convex function using RMPS, it is sufficient to perform only one \textit{run}. RMPS can be made even faster for solving convex problem by increasing the \textit{step decay rate} parameter. In appendix C another comparative performance study of RMPS, GA and SA is considered based on minimizing convex objective functions. In Table S2 of the appendix, it is shown that we obtain up to 40 and 92 folds improvement over GA and SA using modified version of RMPS which is specially designed for convex minimization. 

\begin{table}[h]
	\centering
	\resizebox{0.99\columnwidth}{!}{%
		\begin{tabular}{@{}llllllc@{}}
			\toprule
			\multicolumn{1}{c}{Function names} & \multicolumn{1}{c}{\begin{tabular}[c]{@{}c@{}}RMPS\\ (min. value)\end{tabular}} & \multicolumn{1}{c}{\begin{tabular}[c]{@{}c@{}}GA\\ (min. value)\end{tabular}} & \multicolumn{1}{c}{\begin{tabular}[c]{@{}c@{}}SA\\ (min. value)\end{tabular}} & \multicolumn{1}{c}{\begin{tabular}[c]{@{}c@{}}RMPS\\ (avg. time)\end{tabular}} & \multicolumn{1}{c}{\begin{tabular}[c]{@{}c@{}}GA\\ (avg. time)\end{tabular}} & \begin{tabular}[c]{@{}c@{}}SA\\ (avg. time)\end{tabular} \\ \midrule
			Ackley Function & 0.00012 & 4.95E-06 & 2.46E-06 & 0.072 & 0.312 & 0.548 \\
			Bukin Function N6 & 0.101226 & 0.021039 & 0.00986 & 0.053 & 0.255 & 0.505 \\
			Cross-in-Tray Function & -2.06261 & -2.06261 & -2.06257 & 0.063 & 0.156 & 0.404 \\
			Drop-Wave Function & -1 & -1 & -0.93625 & 0.053 & 0.153 & 0.571 \\
			Eggholder Function & -959.641 & -956.918 & -888.949 & 0.062 & 0.256 & 0.505 \\
			Gramacy \& Lee (2012) Function (n=1) & -0.86901 & -0.86901 & -0.86901 & 0.016 & 0.147 & 0.249 \\
			Griewank Function & 2.25E-07 & 0.007396 & 0.007525 & 0.057 & 0.166 & 0.501 \\
			Holder Table Function & -19.2085 & -19.2085 & -19.2085 & 0.074 & 0.163 & 0.51 \\
			Langermann Function & -4.15581 & -4.15581 & -4.15486 & 0.087 & 0.176 & 0.479 \\
			Levy Function & 9.9E-11 & 4.64E-12 & 1.46E-06 & 0.084 & 0.166 & 0.549 \\
			Levy Function N 13 & 1.66E-10 & 4.56E-11 & 7.47E-07 & 0.075 & 0.182 & 0.401 \\
			Rastrigin Function & 8.65E-09 & 4.79E-10 & 1.42E-05 & 0.069 & 0.179 & 0.476 \\
			Schaffer Function N2 & 1.66E-11 & 1.39E-13 & 0.005998 & 0.062 & 0.184 & 0.425 \\
			Schaffer Function N4 & 0.292579 & 0.295289 & 0.298868 & 0.288 & 0.164 & 0.375 \\
			Schwefel Function & 2.55E-05 & 2.55E-05 & 118.4384 & 0.068 & 0.261 & 0.503 \\
			Shubert Function & -186.731 & -186.731 & -186.731 & 0.052 & 0.173 & 0.439 \\
			Bohachevsky Functions 1 & 3.98E-07 & 7.49E-12 & 8.74E-09 & 0.076 & 0.202 & 0.54 \\
			Bohachevsky Functions 2 & 0.218313 & 1.63E-10 & 2.24E-07 & 0.07 & 0.201 & 0.449 \\
			Bohachevsky Functions 3 & 9.01E-08 & 8.55E-11 & 1.65E-06 & 0.062 & 0.209 & 0.471 \\
			Perm Function 0, n, $\beta(=10)$ & 1.42E-08 & 3.14E-09 & 8.68E-08 & 0.057 & 0.207 & 0.472 \\
			Rotated Hyper-Ellipsoid Function & 1.07E-08 & 2.75E-12 & 3.62E-08 & 0.072 & 0.18 & 0.589 \\
			Sphere Function & 4.36E-11 & 1.99E-11 & 7.88E-08 & 0.054 & 0.162 & 0.429 \\
			Sum of Different Powers Function & 8.32E-13 & 2.81E-15 & 1.05E-06 & 0.068 & 0.149 & 0.486 \\
			Sum Squares Function & 6.54E-11 & 5.16E-13 & 5.98E-08 & 0.072 & 0.165 & 0.465 \\
			Trid Function & -2 & -2 & -2 & 0.047 & 0.158 & 0.411 \\
			Booth Function & -5.3E+07 & -5.3E+07 & -5.3E+07 & 0.031 & 0.227 & 0.299 \\
			Matyas Function & 2.27E-11 & 2.18E-13 & 5.98E-06 & 0.062 & 0.183 & 0.457 \\
			McCormick Function & -1.91322 & -1.91322 & -1.91322 & 0.066 & 0.157 & 0.42 \\
			Power Sum Function (n=4) & 2.06E-05 & 0.001365 & 0.000538 & 3.008 & 0.825 & 0.892 \\
			Zakharov Function & 1.13E-10 & 6.55E-12 & 4.5E-07 & 0.071 & 0.173 & 0.497 \\
			Three-Hump Camel Function & 4.61E-11 & 2.92E-11 & 6.57E-07 & 0.067 & 0.167 & 0.499 \\
			Six-Hump Camel Function & -1.03163 & -1.03163 & -1.03163 & 0.07 & 0.17 & 0.469 \\
			Dixon-Price Function & 1.6E-10 & 2.71E-10 & 6.29E-08 & 0.054 & 0.191 & 0.398 \\
			Rosenbrock Function & 6.57E-06 & 0.000567 & 0.000104 & 1.202 & 0.416 & 0.459 \\
			De Jong Function N5 & 0.998004 & 0.998004 & 0.998004 & 0.078 & 0.165 & 0.483 \\
			Easom Function & -1 & -1 & -1.5E-09 & 0.053 & 0.157 & 0.262 \\
			Michalewicz Function & -1.8013 & -1.8013 & -1.8013 & 0.079 & 0.163 & 0.464 \\
			Beale Function & -0.15509 & -0.15509 & -0.15509 & 0.072 & 0.186 & 0.41 \\
			Branin Function & 0.397887 & 0.397887 & 0.397888 & 0.066 & 0.17 & 0.477 \\
			Colville Function (n=4) & -469.313 & -469.313 & -469.313 & 0.204 & 0.251 & 1.177 \\
			Forrester et al. (2008) Function & -6.02074 & -6.02074 & -6.02074 & 0.017 & 0.154 & 0.182 \\
			Goldstein-Price Function & 3 & 3 & 3 & 0.078 & 0.185 & 0.49 \\
			Perm Function n,$\beta(=0.5)$ & 4.04E-11 & 2.34E-10 & 9.22E-07 & 0.067 & 0.24 & 0.359 \\
			Powell Function (n=4) & 4.48E-07 & 2.26E-05 & 5.69E-05 & 0.609 & 0.564 & 0.991 \\
			Styblinski-Tang Function & -78.3323 & -78.3323 & -78.3323 & 0.078 & 0.16 & 0.453 \\ \bottomrule
	\end{tabular}}
	\caption{Comparison of minimum values achieved and the average computation time (in seconds, computed in MATLAB R2014a) for minimizing benchmark functions with RMPS, GA and SA starting from 10 starting points in each cases. The dimension of all the problems are 2 unless the value of the dimension (i.e., $n$) is mentioned with the name of the function. Please refer to the appendix for the domain of search for each function.}
	\label{comparison_table_1}
\end{table}

\subsection{High-dimensional benchmark function optimization}
\label{high_dim}
To compare the performance of RMPS with that of GA and SA, we consider 100 and 1000 dimensional Ackley's function, Griewank function, Rastrigin function, Schwefel function, Sphere function and the Sum of square function (\cite{Jamil2013}). The search domains for minimization of the above-mentioned functions are taken to be as follows $[-5,5]^n$, $[-10,10 ]^n$ , $[-5.12,5.12]^n$, $[-500, 500]^n$, $[-5.12,5.12]^n$ and $[-5.12,5.12]^n$ respectively. Schwefel function attains the global minimum value 0 at $\mathbf{x}=(420.9687,$ $ \ldots, 420.9687)$  within the corresponding domain while rest attain the global minimum value 0 at the origin. While minimizing a particular function on a given domain, 10 distinct randomly generated initial points are considered and that same set of 10 starting points are used for GA, SA and RMPS. For GA and SA, the smallest of the 10 obtained objective values are noted along with average computation times (Table \ref{table_high_dim}). For RMPS, both the maximum and the minimum of the 10 obtained objective function minimums and the average computation times are noted down for each case. In Table \ref{table_high_dim}, it is noted that RMPS generally outperforms GA and SA. For all cases it is observed that the maximum of the minimums of the obtained solutions by RMPS (over 10 different starting point cases) is better than the minimum values obtained for GA and SA. A noticeable improvement using RMPS over GA and SA is visible especially in the case of solving the 1000 dimensional problems. Using RMPS we get up to 32 and 368 folds improvement in computation time over GA (in case of 100-dimensional Sum squares function) and over SA (in case of 1000-dimensional Sum squares function) respectively. It is also noted that the minimum and the maximum of the 10 obtained minimums obtained by RMPS are quite close. It demonstrates that the performance of RMPS, in general, is less dependent on the starting point.

In Figure \ref{Das_benchmark}, a comparative study of RMPS, GA and SA is made based on the improvement of the value of the 1000 dimensional objective functions in first 30 minutes starting from same initial points. For all the functions except Sum squares function, the objective values are plotted in absolute scale, while for Sum squares function, the objective values are plotted in natural $\log$ scale. It is observed that RMPS outperforms GA and SA by a large margin.  

\begin{table}[]
	\centering
	\resizebox{0.99\columnwidth}{!}{%
		\begin{tabular}{@{}lccccccc@{}}
			\toprule
			Functions & \begin{tabular}[c]{@{}c@{}}RMPS\\ (min)\end{tabular} & \begin{tabular}[c]{@{}c@{}}RMPS\\ (max)\end{tabular} & \begin{tabular}[c]{@{}c@{}}GA\\ (min)\end{tabular} & \begin{tabular}[c]{@{}c@{}}SA\\ (min)\end{tabular} & \begin{tabular}[c]{@{}c@{}}RMPS\\ (avg. time)\end{tabular} & \begin{tabular}[c]{@{}c@{}}GA\\ (avg. time)\end{tabular} & \begin{tabular}[c]{@{}c@{}}SA\\ (avg. time)\end{tabular} \\ \midrule
			Ackley's ($n =100$) & 1.03E-05 & 1.17E-05 & 3.75E-04 & 7.87E-00 & 3.5 & 68.1 & 54.0 \\
			Griewank ($n =100$) &  7.06E-11 & 1.17E-05 & 1.00E-03 & 1.45E-00 & 5.3 & 30.0 & 46.3 \\
			Rastrigin ($n =100$) &  3.65E-07 & 4.14E-07  & 4.14E-07 & 5.01E+02 & 20.1 & 62.9 & 51.8 \\
			Schwefel ($n =100$) & 1.27E-03 &  1.27E-03 & 7.86E+03 & 1.93E+04 & 43.7 & 58.9 & 65.9 \\
			Sphere ($n =100$) & 6.93E-10 & 8.91E-10 & 4.49E-04 & 8.02+01 & 1.9 & 47.3 & 49.8 \\
			Sum squares ($n =100$) & 3.45E-08 & 4.62E-08 & 3.83E-04 & 1.63E-00 & 2.2 & 72.2 & 96.2 \\
			Ackley's ($n =1000$) & 1.09E-05 & 1.12E-05 & 4.42E-00 & 9.92E-00 & 534.5 & 5865.7 & 5984.7 \\
			Griewank ($n =1000$) &  9.98E-11 & 1.48E-10 & 2.49E-00 & 7.45E-00 & 655.7 & 2580.3 & 4573.0 \\
			Rastrigin ($n =1000$) &  3.86E-06 & 3.95E-06 & 2.79E+03 & 7.63E+03 &  3938.4 & 6391.2 & 5591.7 \\
			Schwefel ($d =1000$) & 1.27E-02 & 1.27E-02 & 2.15E+05 & 1.95E+05 &  8664.6 & 9646.3 & 9448.8 \\
			Sphere ($n =1000$) & 7.77E-09	& 8.36E-09 & 1.13E+03 & 1.70E+03 &  225.5 & 3971.4 & 7920.2 \\
			Sum squares ($n =1000$) &  3.89E-06 &  4.21E-06 & 1.86E+05 & 1.34E-00 &  256.6 & 4805.3 & 94542.4 \\ \bottomrule
	\end{tabular}}
	\caption{Comparative study of RMPS, GA and SA for minimizing 100 and 1000 dimensional Ackley, Griewank, Rastrigin, Schwefel, Sphere and Sum squares functions. Benchmark functions are minimized starting from 10 randomly generated starting points from the corresponding domains using RMPS, GA and SA. The minimum of the 10 obtained minimum values of the objective functions is noted down for each method. For RMPS, also the maximum of these 10 obtained minimum objective function values is noted down as well. Average computation times over 10 starting point scenarios for each method is noted down (in seconds). RMPS outperforms other methods and the closeness of the min. and max. of the 10 obtained function minimum values (by RMPS, corresponding to 10 different starting points) implies the performance of RMPS is less dependent on the starting point, in general.}
	\label{table_high_dim}
\end{table}

\begin{figure} 
		\centering
		\includegraphics[width=0.99\linewidth]{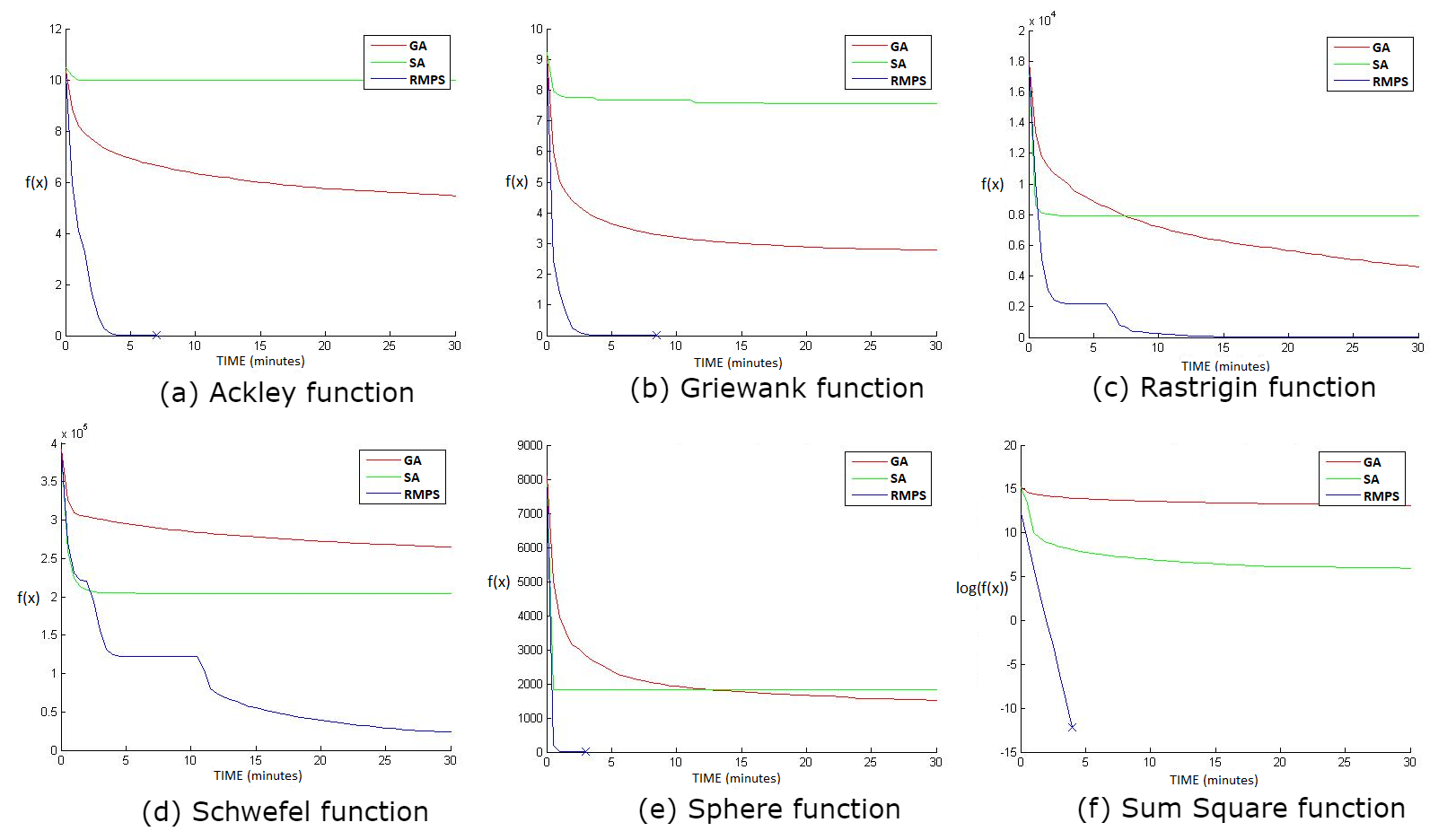} 
		\caption{Comparison of time required for minimizing 1000 dimensionl (a) Ackley (b) Griewank (c) Rastrigin (d) Schwefel (e) Sphere (f) Sum Square functions using Genetic Algorithm (GA), Simulated Annealing (SA) and Recursive Modified Pattern Search (RMPS). Objective function values where noted down every 10 seconds over the period of first 30 minutes. Starting point was generated randomly and same starting point is used for all considered algorithms.} 
		\label{Das_benchmark} 
\end{figure}

\subsection{Boundary point solution cases}
\label{boundary}
It is interesting to gauge the performances of the blackbox optimization methods when the true solution lies at a boundary point of the restricted domain. So we take the domains of Ackley's function, Griewank function, Rastrigin function, Schwefel function, Sphere function and the Sum of square function to be $[0,5]^n$, $[0,10 ]^n$ , $[0,5.12]^n$, $[0, 420.97]^n$, $[0,5.12]^n$ and $[0,5.12]^n$ respectively. Note that, in each case, the true solution is a boundary point. The results (based on 10 randomly generated starting points, similar to the previous cases) are provided in Table S3 of appendix. It is noted that in this case also RMPS generally outperforms GA ans SA. Using RMPS we get up to 40 and 77 folds improvement in computation time over GA (in case of 100-dimensional Sum squares function) and over SA (in case of 1000-dimensional Sum squares function) respectively. Again, it is noted that the best and the worst solutions obtained by RMPS are quite closer which implies that the results obtained by RMPS are less dependent on the starting points, in general.

\subsection{Minimizing 5000 dimensional benchmark functions}
\label{5000}

In Table \ref{table_5000}, we note down the maximum and minimum values of the obtained solutions after minimizing the 5000 dimensional objective functions with RMPS starting from 3 randomly generated initial points. Along with the case where the true global minimum is an interior point on the domain, we also consider the scenario where the true global minimum is a boundary point. The average computation time in each case is also noted down in Table \ref{table_5000}. Both the best and the worst solution cases out of these 3 starting point scenarios are noted. In this case also, it is observed that the obtained best and the worst solutions are quite close to each other, for both boundary and non-boundary cases. For GA and SA, due to unreasonable amount of required computation time in this case scenario, we avoid including their performance results in this paper. 

\begin{table}[]
	\centering
	\centering
	\resizebox{0.99\columnwidth}{!}{%
		\begin{tabular}{@{}ccccccc@{}}
			\toprule
			\multirow{2}{*}{Functions} & \multicolumn{3}{c}{\begin{tabular}[c]{@{}c@{}}Non-Boundary\\ solution case\end{tabular}} & \multicolumn{3}{c}{\begin{tabular}[c]{@{}c@{}}Boundary\\ solution case\end{tabular}} \\ \cmidrule(l){2-7} 
			& \begin{tabular}[c]{@{}c@{}}Min.\\ value\end{tabular} & \begin{tabular}[c]{@{}c@{}}Max.\\ value\end{tabular} & \begin{tabular}[c]{@{}c@{}}Avg. time\\ (minutes)\end{tabular} & \begin{tabular}[c]{@{}c@{}}Min.\\ value\end{tabular} & \begin{tabular}[c]{@{}c@{}}Max.\\ value\end{tabular} & \begin{tabular}[c]{@{}c@{}}Avg. time\\ (minutes)\end{tabular} \\ \midrule
			Ackley's & 1.75E-05 & 1.78E-05 & 697.53 & 1.07E-05 & 1.11E-05 & 681.47 \\
			Griewank & 3.89E-08 & 4.53E-08 & 904.79 & 9.47E-09 & 1.32E-08 & 568.58 \\
			Rastrigin & 7.87E-03 & 7.87E-03 & 417.03 & 3.67E-06 & 3.75E-06 & 720.60 \\
			Schwefel & 3.64E-01 & 3.87E-01 & 5862.39 & 6.37E-02 & 6.37E-02 & 4917.78 \\
			Sphere & 1.01E-05 & 1.01E-05 & 342.13 & 4.02E-08 & 4.11E-08 & 365.64 \\
			Sum squares & 9.79E-05 & 9.94E-05 & 544.64 & 9.91E-05 & 1.03E-04 & 428.12 \\ \bottomrule
	\end{tabular}}
	\caption{Performance evaluation of RMPS, GA and SA for minimizing 5000 dimensional Ackley, Griewank, Rastrigin, Schwefel, Sphere and Sum squares functions. Benchmark functions are minimized starting from 3 randomly generated starting points from the corresponding domains using RMPS, GA and SA. Two scenarios are considered, namely boundary solution case and interior point (non-boundary) solution case. The minimum and the maximum of the 3 obtained minimum values of the objective functions are noted down. Average computation times over 3 starting point scenarios for RMPS is noted down (in minutes). The closeness of the min. and max. of the 3 obtained function minimum values (by RMPS, corresponding to 3 different starting points) implies the performance of RMPS is less dependent on the starting point, in general.}
	\label{table_5000}
\end{table}

\section{Estimating optimal combination of biomarker from Alzheimer's disease data}
\label{sec_HUM}
In precision medicine, the patients are often classified into multiple disease categories based on the values of the biomarkers or pathological test results. Recently, several articles (e.g., \cite{Pepe2000,Pepe2006,Maiti2022}) proposed classifying the subjects into multiple classes using the area under  the receiver operating characteristic (ROC) curve (AUC) as it requires less distributional assumptions compared to a few other comparable approaches. Consider there are $H$ outcome categories and $\{\mathbf{X}_j\}_{j=1}^H$ denote the values of $d$ dimensional biomarkers coming from corresponding classes. Suppose $X_j \sim F_j$ where $F_j$ denotes some continuous multi-variate distribution. Then the linear score obtained by combining the biomarkers for class $j$ is given by $\boldsymbol{\beta}^T\mathbf{X}_j = \sum_{k=1}^d\beta_kX_{jk}$, where $\boldsymbol{\beta}$ denotes the optimal coefficient for combining biomarkers. Assuming higher linear score corresponds to higher disease/class category, the diagnostic accuracy based on thses linear scores can be measured by maximizing the Hypervolume Under Manifold (HUM) given by $D(\boldsymbol{\beta}) = P(\boldsymbol{\beta}^T\mathbf{X}_H> \ldots > \boldsymbol{\beta}^T\mathbf{X}_1)$ (\cite{Li2008}).

Since $D(\boldsymbol{\beta})$ cannot be maximized directly, one alternative to estimate the combination vector is by maximizing it's empirical estimate (\cite{Zhang2011}). Suppose the sample values are given by $\{\mathbf{X}_{ji_j}: j = 1,\ldots, H, i_j = 1,\ldots,n_j\}$, where $n_j$ denotes the sample size of $j$-th class; the total sample size is $n = \sum_{j=1}^Hn_j$. Then the empirical estimate of HUM (EHUM) is given by 
\begin{align*}
 {D}_{E}(\boldsymbol{\beta}) = \dfrac{1}{n_{1}n_{2} \cdots n_{H}} \displaystyle\sum_{i_{1}=1}^{n_{1}}\displaystyle\sum_{i_{2}=1}^{n_{2}} \cdots \displaystyle\sum_{i_{H}=1}^{n_{H}}I(\boldsymbol{\beta}^{T}\mathbf{X}_{i_{H}}>\boldsymbol{\beta}^{T}\mathbf{X}_{i_{(H-1)}}>\cdots>\boldsymbol{\beta}^{T}\mathbf{X}_{i_{1}}).
\end{align*}
However there remains a few computational issues in order to maximize EHUM function as follows. Firstly, evaluation of objective function can be very expensive for higher number of classes and large sample scenarios. Secondly, the EHUM objective function is a sum over step functions which make it discontinuous and not suitable for derivative based optimization techniques. Lastly, the possible multi-modal nature of EHUM makes is even harder to maximize.
To alleviate the computational burden of maximizing EHUM function, \cite{Hsu2016} proposed to obtain the optimal combination vector by maximizing
$$P_{A}(\boldsymbol{\beta}) =\frac{1}{H-1} \displaystyle\sum_{j=1}^{H-1}P(\boldsymbol{\beta}^{T}\mathbf{X}_{j+1}>\boldsymbol{\beta}^{T}\mathbf{X}_{j})$$ (see Section D of the appendix for the detailed formulation). For the rest of the paper, we name it Upper and Lower Bound Approach (ULBA). Compared to EHUM, although ULBA objective function computationally less expensive to evaluate, rest of the challenges for maximizing EHUM (e.g., discontinuity, possible multi-modal nature) also applies for ULBA as well.

The global optimization of discontinuous and (possibly) multi-modal objective function EHUM with respect to $\boldsymbol{\beta}$ becomes more challenging as the number of classes and sample size increase. So instead of estimating the whole optimal combination vector simultaneously, \cite{Pepe2006} proposed Step-down algorithm where, the biomarkers are ordered based on their individual EHUM values. Then the biomarker with highest individual EHUM value is included. Afterwards, sequentially one biomarker is included at a time and their coefficients are calculated solving univariate optimization one at a time. A brief description of step-down algorithm is provided in Section D of the appendix. Similar technique can also be used to optimize ULBA criteria as well. In short, Step-down alleviates the computational burden of solving a multivariate optimization problem by breaking it into sequence of univariate optimization problems. We use RMPS algorithm to maximize EHUM and ULBA criteria. Note that, $D_E(\boldsymbol{\beta})$ and $P_A(\boldsymbol{\beta})$ are not identifiable. Therefore, while optimizing using RMPS, we set the value of the first biomarker to be 1 or -1 (based on which coefficient results in higher individual EHUM or ULBA value) and then rest $d-1$ biomarkers are estimated. We consider the upper and lower bounds of the domain of the possible values of the biomarker coefficients to be 10 and -10 for all coordinates respectively. 

\begin{figure}[ht]
	\centering
	\includegraphics[width=0.99\textwidth]{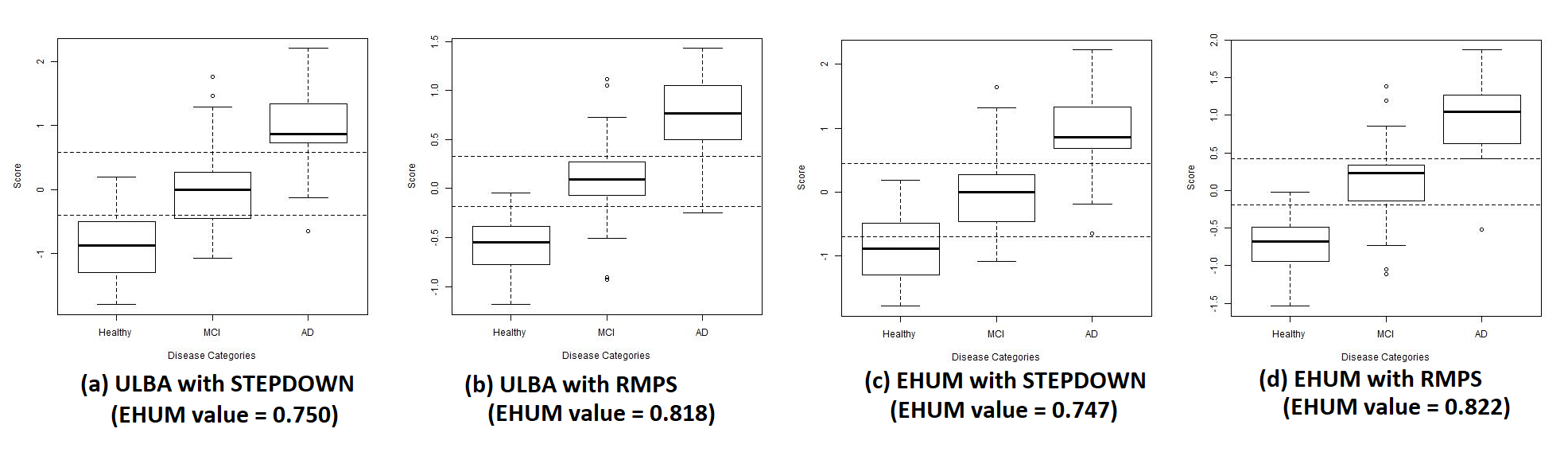}
	\caption{Based on Alzheimer's dataset, Upper and Lower Bound Approach (ULBA) ($P_A(\boldsymbol{\beta})$) and Empirical Hypervolume Under Manifol (EHUM) ($D_E(\boldsymbol{\beta})$) objective functions are maximized using Stepdown and RMPS algorithms. The boxplot of the scores of the true classes are shown. The dotted lines represent the Youden's Index calculated from the dataset. EHUM objective function values are also noted down at the calculated for the optimal biomarker combination vector obtained for each of the 4 cases. }
	\label{ehum}
\end{figure}

We consider a dataset (\cite{Luo2012}) on 14 biomarkers (neuropsychometric tests) whose values are available 108 Alzheimer's patients belonging to 3 different stages of the disease. 14 markers are given as follows: global (\textit{factor1}), temporal  (\textit{ktemp}), parietal  (\textit{kpar}), frontal  (\textit{kfront}), logical memory (\textit{zpsy004}), digital span forward (\textit{zpsy005}), digital span backward (\textit{zpsy006}), information (\textit{zinfo}), two measures of visual retention (\textit{zbentc}, \textit{zbentd}), Boston naming (\textit{zboston}), mental control (\textit{zmentcon}), word fluency (\textit{zworflu}), and associate learning (\textit{zassc}). Due to higher correlation among  \textit{factor1}, \textit{ktemp} and \textit{zpsy004} (\cite{Luo2012}), we only consider \textit{ktemp} and drop other two biomarkers among these three biomarkers, with a total of 12 biomarkers.

We maximize both ULBA and EHUM objective functions using step-down and RMPS algorithms. Stepdown algorithm is coded in MATLAB and we use \texttt{fminsearch} function of solving the sequence of univariate optimization problems. RMPS is coded in MATLAB as well. In Figure \ref{ehum} we provide the boxplots of the linear scores obtained for three different stages of Alzheimer's patients in each scenario. EHUM objective function values are also noted down correspoding to the obtained optimal biomarker combination vector for considered methods. It is noted that the EHUM values obtained for optimal combination vector by maximizing $D_E(\boldsymbol{\beta})$ and $P_A(\boldsymbol{\beta})$  using RMPS are higher than that of Step-down for both cases. It is also noted that the classification of the linear scores with Youden's Index (\cite{Youden1950}), yields comparatively more non-overlapping cluster using (combination vectors estimated using) RMPS compared to that obtained from stepdown technique. 

\section{Discussion}
\label{sec_discussion}
In this paper we propose a derivative-free blackbox optimization technique using modified pattern search for optimizing any function over hyper-rectangular domain. RMPS is shown to perform upto 40 times and 368 times faster compared to GA and SA respectively. Along with MATLAB code \href{https://github.com/priyamdas2/RMPS}{\textbf{here}}, the R package \texttt{RMPSH} (\cite{Das2020b}) is available on R CRAN. The novel features of the proposed algorithm are as follows:
\begin{itemize}
    \item RMPS is derivative-free, which is crucial for optimizing any discontinuous, complex structured function with intractable derivative values.
    \item While optimizing $n$ dimensional blackbox function, upto $2n$ parallel threads can be used for computation. In Table S4 of the appendix, based on a case study, it is shown how usage of parallel computing can make RMPS even faster when the objective function is computationally very expensive.
    \item Theoretically RMPS is shown to yield the global solution over the compact space if the objective function is locally convex and differentiable over that space.
    \item RMPS is shown to be well salable for higher dimensional (upto 5000 dimensional problem) optimization problems and it yields better results in less computational time compared to other blackbox optimization techniques, in general.
    \item By simulation experiments it is shown that the performance of RMPS is, in general, equally good for all considered starting points and it's performance does not depend a lot on the starting point (shown for upto 5000 dimensional problems, see Table \ref{table_high_dim} and \ref{table_5000}).
    \item RMPS is shown to perform well even when the true solution is a boundary point. 
    \item When the objective function is known to be convex beforehand, RMPS can be made further scalable by changing the values of a coupe of tuning parameters.
\end{itemize}

Using RMPS we estimate the optimal coefficients for combining the neuropsychometric test results from the subjects with Alzheimer's disease to classify them into three different stages of the disease. RMPS is used to maximize the estimates of HUM which are discontinuous and (possibly) multi-modal in nature. It is shown that using RMPS better classification can be performed compared to existing step-sown algorithm. In order to praise the ability of RMPS to use parallel threads, we consider another case study where matrix completion is performed with Smoothly Clipped Absolute Deviation (SCAD) penalty (see Section D, Figure S3, S4 of the appendix). Due to very expensive nature of the objective function, benefits of using parallel computing is more visible for that case study (see Table S4 of the appendix). 

Variations of RMPS algorithm has been already used successfully on several articles (e.g., \cite{Das2023,Das2023b,Das2022,Das2021,Das2017b}) for solving high-dimensional complex optimization problems arising in different fields of statistics. In future, the idea of RMPS can be extended for optimizing blackbox functions on several possible constrained spaces e.g., over the space of positive definite matrices, orthonormal matrices.

\begin{acknowledgements}
I would really like to thank Dr. Rudrodip Majumdar, Dr. Debraj Das and Dr. Subhashis Ghoshal for helping me with their valuable comments to improve the first version of this draft.
\end{acknowledgements}

\section*{Conflict of interest}
The authors declare that they have no conflict of interest.

\bibliographystyle{spmpsci}      
\bibliography{bibli}   

\begin{thebibliography}{10}
\providecommand{\url}[1]{{#1}}
\providecommand{\urlprefix}{URL }
\expandafter\ifx\csname urlstyle\endcsname\relax
  \providecommand{\doi}[1]{DOI~\discretionary{}{}{}#1}\else
  \providecommand{\doi}{DOI~\discretionary{}{}{}\begingroup
  \urlstyle{rm}\Url}\fi

\bibitem{Audet2014}
Audet, C.: A survey on direct search methods for blackbox optimization and
  their applications.
\newblock Mathematics without boundaries: Surveys in interdisciplinary research
  \textbf{chapter 2}, 31--56 (2014)

\bibitem{Audet2008}
Audet, C., Bechard, V., Digabel, S.L.: Nonsmooth optimization through mesh
  adaptive direct search and variable neighborhood search.
\newblock Journal of Global Optimization \textbf{41}(2), 299--318 (2008)

\bibitem{Audet2006}
Audet, C., Dennis, J.: Mesh adaptive direct search algorithms for constrained
  optimization.
\newblock SIAM Journal on Optimization \textbf{17}(1), 188--217 (2006)

\bibitem{Audet22008}
Audet, C., Jr., J.D., Digabel, S.L.: Parallel space decomposition of the mesh
  adaptive direct search algorithm.
\newblock SIAM Journal on Optimization \textbf{19}(3), 1150--1170 (2008)

\bibitem{Bethke1980}
Bethke, A.D.: Genetic algorithms as function optimizers  (1980).
\newblock \urlprefix\url{https://api.semanticscholar.org/CorpusID:60965631}

\bibitem{Boggs1996}
Boggs, P., Tolle, J.: Sequential quadratic programmings.
\newblock Acta Numerica pp. 1--52 (1996)

\bibitem{Byrd2000}
Byrd, R., Gilbert, J., Nocedal, J.: A trust region method based on interior
  point techniques for nonlinear programming.
\newblock Mathematical Programming \textbf{89}(1), 149--185 (2000)

\bibitem{Candes2009}
Candes, E., Recht, B.: Exact matrix completion via convex optimization.
\newblock Foundations of Computational Mathematics \textbf{9}, 717--772 (2009)

\bibitem{Candes2010}
Candes, E., Tao, T.: The power of convex relaxation: Near-optimal matrix
  completion.
\newblock IEEE transactions on information theory \textbf{56}(5), 2053--2080
  (2010)

\bibitem{Conn2009}
Conn, A., Scheinberg, K., Vicente, L.: Introduction to derivative-free
  optimization.
\newblock Mathematics without boundaries: Surveys in interdisciplinary
  research, MOS-SIAM Series on Optimization, SIAM  (2009)

\bibitem{Custodio2015}
Custodio, A., Madeira, J.: Glods: Global and local optimization using direct
  search.
\newblock Journal of Global Optimization \textbf{62}(1), 1--28 (2015)

\bibitem{Das2021}
Das, P.: Recursive modified pattern search on high-dimensional simplex : a
  blackbox optimization technique.
\newblock The Indian Journal of Statistics - Sankhya B \textbf{83}, 440--483
  (2021)

\bibitem{Das2020b}
Das, P., De, D.: Rmpsh: A r package for recursive modified pattern search on
  hyper-rectangle.
\newblock R CRAN \textbf{\url{https://CRAN.R-project.org/package=RMPSH}} (2020)

\bibitem{Das2022}
Das, P., De, D., Maiti, R., Kamal, M., Hutcheson, K.A., Fuller, C.D.,
  Chakraborty, B., Peterson, C.B.: Estimating the optimal linear combination of
  predictors using spherically constrained optimization.
\newblock BMC Bioinformatics \textbf{23(3)}(436) (2022)

\bibitem{Das2017b}
Das, P., Ghosal, S.: Analyzing ozone concentration by bayesian
  spatio‐temporal quantile regression.
\newblock Environmetrics \textbf{28}(4), e2443 (2017)

\bibitem{Das2023}
Das, P., Sen, D., De, D., Hou, J., Abad, Z., Kim, N., Xia, Z., Cai, T.:
  Clustering sequence data with mixture markov chains with covariates using
  multiple simplex constrained optimization routine (msicor).
\newblock Journal of Computational and Graphical Statistics \textbf{(to
  appear)} (2023).
\newblock \urlprefix\url{https://doi.org/10.1080/10618600.2023.2257258}

\bibitem{Das2023b}
Das, P., Weisenfeld, D., Dahal, K., De, D., Feathers, V., Coblyn, J.,
  Weinblatt, M., Shadick, N., Cai, T., Liao, K.: Utilizing biologic
  disease-modifying anti-rheumatic treatment sequences to subphenotype
  rheumatoid arthritis.
\newblock Arthritis Research and Therapy \textbf{25}(1), 1--7 (2023)

\bibitem{Digabel2011}
Digabel, S.L.: Algorithm 909: Nomad: Nonlinear optimization with the mads
  algorithm.
\newblock ACM Transactions on Mathematical Software \textbf{37}(4(44)), 1--15
  (2011)

\bibitem{Fan2001}
Fan, J., Li, R.: Variable selection via nonconcave penalized likelihood and its
  oracle properties.
\newblock Journal of the American Statistical Association \textbf{96}(456),
  1348--1360 (2001)

\bibitem{Fermi1952}
Fermi, E., Metropolis, N.: Numerical solution of a minimum problem. los alamos
  unclassified report la–1492.
\newblock Los Alamos National Laboratory, Los Alamos, USA  (1952)

\bibitem{Fraser1957}
Fraser, A.: Simulation of genetic systems by automatic digital computers i.
  introduction.
\newblock Australian Journal of Biological Sciences \textbf{10}, 484--491
  (1957)

\bibitem{Geris2012}
Geris, L.: Computational Modeling in Tissue Engineering.
\newblock Springer (2012)

\bibitem{Goldberg1989}
Goldberg, D.: Genetic Algorithms in Search, Optimization, and Machine Learning.
\newblock Operations Research Series, Addison-Wesley Publishing Company (1989)

\bibitem{Goodner2012}
Goodner, J., Tsianos, G., Li, Y., Loeb, G.: Biosearch: A physiologically
  plausible learning model for the sensorimotor system.
\newblock Proceedings of the Society for Neuroscience Annual Meeting  (2012)

\bibitem{Granville1994}
Granville, V., Krivanek, M., Rasson, J.P.: Simulated annealing: A proof of
  convergence.
\newblock IEEE Transactions on Pattern Analysis and Machine Intelligence
  \textbf{16}, 652--656 (1994)

\bibitem{Hsu2016}
Hsu, M., Chen, Y.: Optimal linear combination of biomarkers for multi-category
  diagnosis.
\newblock Statistics in Medicine \textbf{35}(2), 202--213 (2016)

\bibitem{Huyer1999}
Huyer, W., Neumaier, A.: Global optimization by multilevel coordinate search.
\newblock Journal of Global Optimization \textbf{14}, 331--355 (1999)

\bibitem{Jamil2013}
Jamil, M., Yang, X.: A literature survey of benchmark functions for global
  optimization problems.
\newblock Int. J. of Mathematical Modelling and Numerical Optimisation
  \textbf{4}(2) (2013)

\bibitem{Jones1998}
Jones, D., Schonlau, M., Welch, W.: Efficient global optimization of expensive
  black box functions.
\newblock Journal of Global Optimization \textbf{13}(4), 455--492 (1998)

\bibitem{Kennedy1995}
Kennedy, J., Eberhart, R.: Particle swarm optimization.
\newblock In Proceedings of the IEEE International Conference on Neural
  Networks, Piscataway, NJ, USA pp. 1942--1948 (1995)

\bibitem{Kirkpatrick1983}
Kirkpatrick, S., Gelatt, C., Vecchi, M.: Optimization by simulated annealing.
\newblock Australian Journal of Biological Sciences \textbf{220}(4598),
  671--680 (1983)

\bibitem{Kolda2003}
Kolda, T., Lewis, R., Torczon, V.: Optimization by direct search: New
  perspectives on some classical and modern methods.
\newblock SIAM Review \textbf{45}(3), 385--482 (2003)

\bibitem{Lewis1999}
Lewis, R., Torczon, V.: Pattern search algorithms for bound constrained
  minimization.
\newblock SIAM Journal on Optimization \textbf{9}(4), 1082--1099 (1999)

\bibitem{Li2008}
Li, J., Fine, J.: {ROC} analysis with multiple classes and multiple tests:
  methodology and its application in microarray studies.
\newblock Biostatistics \textbf{9}(3), 566--576 (2008)

\bibitem{Luo2012}
Luo, J., Xiong, C.: Diagtest3grp: An r package for analyzing diagnostic tests
  with three ordinal groups.
\newblock J Stat Softw. \textbf{51}(3), 1--24 (2012)

\bibitem{Maiti2022}
Maiti, R., Li, J., Das, P., Liu, X., Feng, L., Hausenloy, D.J., Chakraborty,
  B.: A distribution-free smoothed combination method to improve discrimination
  accuracy in multi-category classification.
\newblock Statistical Methods in Medical Research \textbf{32}(2), 242--266
  (2023)

\bibitem{Marquardt1963}
Marquardt, D.: An algorithm for least-squares estimation of nonlinear
  parameters.
\newblock Journal of the Society for Industrial and Applied Mathematics
  \textbf{11}, 431--441 (1963)

\bibitem{Martelli2014}
Martelli, E., Amaldi, E.: Pgs-com: A hybrid method for constrained non-smooth
  black-box optimization problems: Brief review, novel algorithm and
  comparative evaluation.
\newblock Computers and Chemical Engineering \textbf{63}, 108--139 (2014)

\bibitem{Martinez2013}
Martinez, J., Sobral, F.: Constrained derivative-free optimization on thin
  domains.
\newblock Journal of Global Optimization \textbf{56}(3), 1217--1232 (2003)

\bibitem{Pepe2006}
Pepe, M., Cai, T., Longton, G.: Combining predictors for classification using
  the area under the receiver operating characteristic curve.
\newblock Biometrics \textbf{62}(1), 221--229 (2006)

\bibitem{Pepe2000}
Pepe, M., Thompson, M.: Combining diagnostic test results to increase accuracy.
\newblock Biostatistics \textbf{1}(2), 123--140 (2000)

\bibitem{Potra2000}
Potra, F., Wright, S.: Interior-point methods.
\newblock Journal of Computational and Applied Mathematics \textbf{4}, 281--302
  (2000)

\bibitem{Tibshirani1996}
Tibshirani, R.: Regression shrinkage and selection via the lasso.
\newblock Journal of the Royal Statistical Society. Series B (Methodological)
  \textbf{58}(1), 267--288 (1996)

\bibitem{Torczon1997}
Torczon, V.: On the convergence of pattern search algorithms.
\newblock SIAM Journal on Optimization \textbf{7}, 1--25 (1997)

\bibitem{Youden1950}
Youden, W.: Index for rating diagnostic tests.
\newblock Cancer \textbf{3}, 32--35 (1950)

\bibitem{Zhang2011}
Zhang, Y., Li, J.: Combining multiple markers for multi-category
  classification: An {ROC} surface approach.
\newblock Australian \& New Zealand Journal of Statistics \textbf{53}(1),
  63--78 (2011)

\end{thebibliography}

%
%

\newpage
\appendix  

\setcounter{table}{0}
\renewcommand{\thetable}{S\arabic{table}}
\setcounter{figure}{0}
\renewcommand{\thefigure}{S\arabic{figure}}

\section*{Appendix A : RMPS and Generalized Pattern Search}
\label{comparison_GPS}
Here we discuss some of the key features of RMPS along with the differences of RMPS with a few existing Pattern Search based optimization techniques (e.g., \cite{Torczon1997}). Firstly, the restart strategy with smaller step-size decay rate is something which is possibly proposed for the first time in the context of Pattern search, to the best of our knowledge. Secondly, unlike algorithm 1 of \cite{Torczon1997}, instead of unconstrained minimization, the proposed algorithm minimizes the black-box function on a hyper-rectangle. In \cite{Fermi1952} and \cite{Torczon1997}, the coordinate-wise jump sizes were kept equal inside an iteration while in the proposed algorithm, the domain of each coordinate being bounded, in every iteration, local-step sizes are modified separately for each coordinates in each direction (positive and negative), as required. In GPS, each coordinate-wise jump step-sizes are evaluated using `exploratory moves algorithm' (see \cite{Torczon1997}) while in the proposed algorithm it is straightforward and does not use `exploratory moves algorithm'. While optimizing a function on a hyper-rectangle, since the domain is transformed into an unit hyper-cube, the global step-size is kept same for each coordinate. So while determining the step-sizes of coordinate-wise movements, the proposed algorithm uses different strategy than the `exploratory moves algorithm'. The most unique feature of the proposed algorithm is the restart strategy as described in the main paper, in details.

\section*{Appendix B : Tuning parameters and their roles}
\label{tuning}
In this section, a brief description of the tuning parameters and their roles are provided which are as follows.
\begin{description}
	\item[$\bullet$ \textit{step decay rate} ($\rho$) :] $\rho$ determines the rate of change of \textit{global step size} at the end of each iteration. . So it is understandable that the value of $\rho$ must be greater than 1. Taking smaller values of $\rho$ will make the decay of step sizes slower, which would allow finer search within the domain at the cost of more computation time.  
    
	\item[$\bullet$ \textit{step size threshold} ($\phi$) :] $\phi$ controls the precision of the solution. This is the minimum possible value that the \textit{global step size} and the \textit{local step sizes} could take. Once the \textit{global step size} reaches below $\phi$, the \textit{run} is terminated. Taking $\phi$ to be smaller results in better precision in the cost of higher computation time.

	\item[$\bullet$ $\textit{tol\_fun}$ :] $\textit{tol\_fun}$ denotes the minimum amount of improvement after an iteration in the solution which is required to keep the value of \textit{global step size} unchanged. In other words, if the differences of solutions obtained in two consecutive iterations is less than \textit{tol\_fun}, the improvement is considered to be `not significant' and in that scenario, the \textit{global step size} is decreased for next iteration for employing a finer search.
    
	\item[$\bullet$ $\textit{tol\_fun\_2}$ :] The second \textit{run} onwards, whenever a \textit{run} terminates, it is checked whether the solution returned by the current \textit{run} is the same or different with the solution returned by the previous \textit{run}. However, to check whether they are exactly equal, they need to be matched up to several decimal places depending on the type of storage variable used by the software. Thus it might result into performance of a lot of extra \textit{runs} just to improve the solution at distant decimal places which might be not of our interest. Therefore, once the euclidean distance of solution points obtained from two consecutive \textit{runs} become less than $\textit{tol\_fun\_2}$, the algorithm is terminated and the final result is returned.
\end{description}

\section*{Appendix C : Additional comparative performance study of RMPS}
The domains of the optimization problems considered in Table 1 of the main paper is provided in Table \ref{domains}. Figures \ref{theory_1} and \ref{theory_2} can be used as an aid for visualization corresponding to the theoretical properties of RMPS shown in Section 3 of the main paper. 
\subsection*{Exploiting convexity}
\label{sec_convex}
The prior knowledge of convexity can be used to save computation time using RMPS. In order to improve computation time for minimizing convex functions, we also consider RMPS taking $\rho=4$ which employs steeper decrease in \textit{global step size} and \textit{local step sizes}. In Table \ref{table_convex}, a comparison study of performances of RMPS, and changed RMPS with the prior knowledge of convexity (RMPS(c), (i.e., \textit{max\_runs}=1, $\rho=4$ with default values of other parameters), GA and SA are provided for minimizing Sphere and Sum squares function for various dimensions starting from 10 randomly generated starting points (under 10 random number generating seeds in MATLAB) in each cases. It is noted that in each cases, RMPS(c) performs faster than RMPS. Also in terms of computation times, using RMPS(c) up to 40 folds improvement is observed compared to GA and up to 92 folds improvement is observed compared to SA. 

In Table \ref{table_high_dim_boundary} we note down the comparative performance results of RMPS, GA and SA when the true solution lies at a boundary point of the domain. 
 
\begin{table}[H]
	\centering
	\resizebox{0.8\columnwidth}{!}{%
		\begin{tabular}{@{}lcc@{}}
			\hline
			\multicolumn{1}{c}{Function names} & Domain & True minimum \\ \hline
			Ackley Function & $[-32.768, 32.768]^2$ & 0 \\
			Bukin Function N6 & $[-15,-5]\times[-3,3]$ & 0 \\
			Cross-in-Tray Function & $[-10, 10]^2$ & -2.0626 \\
			Drop-Wave Function & $[-5.12, 5.12]^2$ & -1 \\
			Eggholder Function & $[-512, 512]^2$ & -959.6407 \\
			Gramacy \& Lee (2012) Function (n=1) & $[0.5,2.5]$ & unknown \\
			Griewank Function & $[-600, 600]^2$ & 0 \\
			Holder Table Function & $[-10, 10]^2$ & -19.2085 \\
			Langermann Function & $[0,10]^2$ & unknown \\
			Levy Function & $[-10, 10]^2$ & 0 \\
			Levy Function N 13 & $[-10, 10]^2$ & 0 \\
			Rastrigin Function & $[-5.12, 5.12]^2$ & 0 \\
			Schaffer Function N2 & $[-100,100]^2$ & 0 \\
			Schaffer Function N4 & $[-100,100]^2$ & 0.2926 \\
			Schwefel Function & $[-500,500]^2$ & 0 \\
			Shubert Function & $[-5.12, 5.12]^2$ & -186.7309 \\
			Bohachevsky Functions 1 & $[-100,100]^2$ & 0 \\
			Bohachevsky Functions 2 & $[-100,100]^2$ & 0 \\
			Bohachevsky Functions 3 & $[-100,100]^2$ & 0 \\
			Perm Function 0, $d(=2)$, $\beta(=10)$ & $[-2, 2]^2$ & 0 \\
			Rotated Hyper-Ellipsoid Function & $[-65.536, 65.536]^2$ & 0 \\
			Sphere Function & $[-5.12, 5.12]^2$ & 0 \\
			Sum of Different Powers Function & $[-1, 1]^2$ & 0 \\
			Sum Squares Function & $[-5.12, 5.12]^2$ & 0 \\
			Trid Function & $[-4, 4]^2$ & unknown \\
			Booth Function & $[-10, 10]^2$ & 0 \\
			Matyas Function & $[-10, 10]^2$ & 0 \\
			McCormick Function & $[-1.5,4]\times[-3,4]$ &  \\
			Power Sum Function $(n=4)$ & $[0,4]^4$ & unknown \\
			Zakharov Function & $[-5, 10]^2$ & 0 \\
			Three-Hump Camel Function & $[-5, 5]^2$ & 0 \\
			Six-Hump Camel Function & $[-3,3]\times[-2,2]$ & -1.0316 \\
			Dixon-Price Function & $[-10, 10]^2$ & 0 \\
			Rosenbrock Function & $[-5, 10]^2$ & 0 \\
			De Jong Function N5 & $[-65.536, 65.536]^2$ & unknown \\
			Easom Function & $[-100,100]^2$ & -1 \\
			Michalewicz Function & $[0, \pi]^2$ & -1.8013 \\
			Beale Function & $[-4.5, 4.5]^2$ & 0 \\
			Branin Function & $[-5,10]\times[0,15]$ & 0.397887 \\
			Colville Function $(n=4)$ & $[-10, 10]^4$ & 0 \\
			Forrester et al. (2008) Function & $[0,1]$ & unknown \\
			Goldstein-Price Function & $[-2, 2]^2$ & 3 \\
			Perm Function $n(=2)$,$\beta(=0.5)$ & $[-2, 2]^2$ & 0 \\
			Powell Function $(n=4)$ & $[-4,5]^4$ & 0 \\
			Styblinski-Tang Function & $[-5, 5]^2$ & -78.3323 \\ \hline
	\end{tabular}}
	\caption{Domains of search regions of the benchmark functions considered in Section 4 in the main paper.}
	\label{domains}
\end{table}

	\begin{figure}[H]
		\centering
		\begin{minipage}[b]{0.45\textwidth}
			\includegraphics[width=\textwidth]{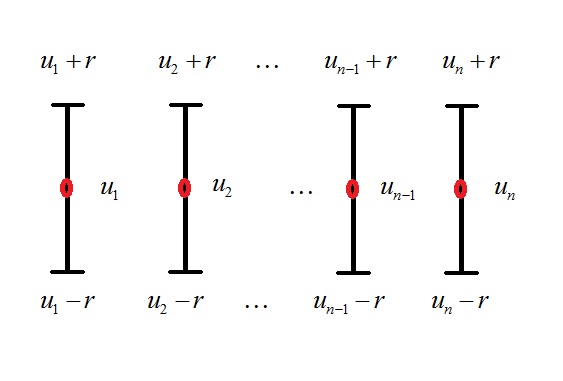}
			\caption{$\mathbf{U} = \prod_{i=1}^n(u_i-r,u_i+r)$.}
			\label{theory_1}
		\end{minipage}
		\begin{minipage}[b]{0.45\textwidth}
			\includegraphics[width=\textwidth]{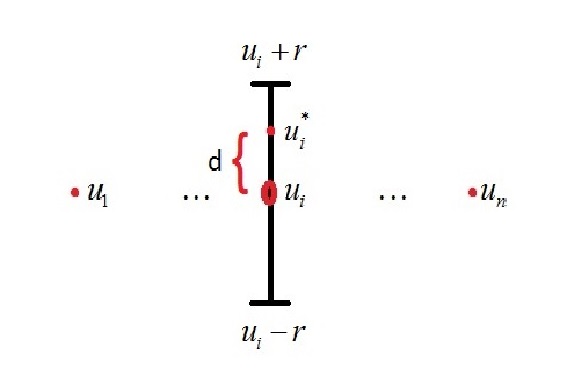}
			\caption{Neighborhood of $u_i$.}
			\label{theory_2}
		\end{minipage}
	\end{figure}

\begin{table}[H]
	\centering
	\resizebox{0.99\columnwidth}{!}{%
		\begin{tabular}{@{}lcccccccc@{}}
			\hline
			Functions & \begin{tabular}[c]{@{}c@{}}RMPS(c)\\ (min value)\end{tabular} & \begin{tabular}[c]{@{}c@{}}RMPS\\ (max value)\end{tabular} & \begin{tabular}[c]{@{}c@{}}GA\\ (min value)\end{tabular} & \begin{tabular}[c]{@{}c@{}}SA\\ (min value)\end{tabular} & \begin{tabular}[c]{@{}c@{}}RMPS(c)\\ (avg. time)\end{tabular} & \begin{tabular}[c]{@{}c@{}}RMPS\\ (avg. time)\end{tabular} & \begin{tabular}[c]{@{}c@{}}GA\\ (avg. time)\end{tabular} & \begin{tabular}[c]{@{}c@{}}SA\\ (avg. time)\end{tabular} \\ \hline
			\begin{tabular}[c]{@{}l@{}}Sphere \\ Function $(n=4)$\end{tabular} & 1.02E-11 & 6.55E-11 & 1.03E-10 & 1.41E-04 & 0.012 & 0.03 & 0.207 & 0.945 \\
			\begin{tabular}[c]{@{}l@{}}Sphere \\ Function $(n=20)$\end{tabular} & 1.17E-10 & 2.22E-10 & 1.22E-06 & 4.12E-00 & 0.054 & 0.116 & 2.363 & 5.462 \\
			\begin{tabular}[c]{@{}l@{}}Sphere \\ Function $(n=40)$\end{tabular} & 2.56E-10 & 3.67E-10 & 2.09E-05 & 2.00E+01 & 0.186 & 0.299 & 6.07 & 9.301 \\
			\begin{tabular}[c]{@{}l@{}}Sphere \\ Function $(n=100)$\end{tabular} & 6.93E-10 & 8.91E-10 & 4.49E-04 & 8.02E+01 & 1.07 & 1.393 & 45.923 & 47.568 \\
			\begin{tabular}[c]{@{}l@{}}Sum squares\\ Function $(n=4)$\end{tabular} & 1.66E-11 & 1.91E-10 & 5.39E-10 & 1.21E-04 & 0.013 & 0.038 & 0.219 & 0.885 \\
			\begin{tabular}[c]{@{}l@{}}Sum squares \\ Function $(n=20)$\end{tabular} & 1.19E-09 & 2.20E-09 & 1.63E-06 & 5.79E-01 & 0.076 & 0.172 & 3.085 & 4.487 \\
			\begin{tabular}[c]{@{}l@{}}Sum squares \\ Function $(n=40)$\end{tabular} & 5.44E-09 & 7.49E-09 & 3.71E-05 & 7.45E-01 & 0.297 & 0.457 & 8.552 & 13.083 \\
			\begin{tabular}[c]{@{}l@{}}Sum squares \\ Function $(n=100)$\end{tabular} & 3.45E-08 & 4.62E-08 & 3.83E-04 & 1.63E-00 & 1.763 & 2.221 & 75.519 & 92.489 \\ \hline
	\end{tabular}}
	\caption{Comparative study of RMPS, RMPS(c)(i.e., \textit{max\_runs}=1, $\rho=4$ with default values of other parameters), GA and SA for solving convex problems Sphere and Sum squares functions for $d=4,20,40,100$ are minimized using RMPS, GA and SA starting from 10 randomly generated starting points from the corresponding domains. The minimum of the 10 obtained minimum values of the objective functions is noted down for each method. For RMPS, also the maximum of these 10 obtained minimum objective function values is noted down as well. Average computation times over 10 starting point scenarios for each method is noted down (in seconds).}
	\label{table_convex}
\end{table}

\begin{table}[H]
	\centering
	\resizebox{0.99\columnwidth}{!}{%
		\begin{tabular}{@{}lccccccc@{}}
			\hline
			Functions & \begin{tabular}[c]{@{}c@{}}RMPS\\ (min)\end{tabular} & \begin{tabular}[c]{@{}c@{}}RMPS\\ (max)\end{tabular} & \begin{tabular}[c]{@{}c@{}}GA\\ (min)\end{tabular} & \begin{tabular}[c]{@{}c@{}}SA\\ (min)\end{tabular} & \begin{tabular}[c]{@{}c@{}}RMPS\\ (avg. time)\end{tabular} & \begin{tabular}[c]{@{}c@{}}GA\\ (avg. time)\end{tabular} & \begin{tabular}[c]{@{}c@{}}SA\\ (avg. time)\end{tabular} \\ \hline
			Ackley's ($n =100$) & 8.55E-06 & 1.16E-05 & 2.29E-00 & 8.75E-00 & 4.8 & 39.5 & 40.0 \\
			Griewank ($n =100$) & 5.24E-11 & 1.23E-02 & 3.28E-04 & 1.46E-00 & 5.6 & 50.2 & 45.0 \\
			Rastrigin ($n =100$) & 4.79E-08 & 9.29E-08 & 2.19E+01 & 6.81E+00 & 3.4 & 56.4 & 64.0 \\
			Schwefel ($n =100$) & 1.27E-03 & 1.27E-03 & 7.36+03 & 1.91+04 & 5.3 & 26.0 & 87.5 \\
			Sphere ($n =100$) & 7.32E-10 & 8.76E-10 & 5.22E-04 & 1.72+02 & 2.5 & 72.4 & 55.8 \\
			Sum squares ($n =100$) & 3.47E-08 & 4.58E-08 & 1.30E-03 & 2.20 + 01 & 2.7 & 108.0 & 66.7 \\
			Ackley's ($n =1000$) & 9.8E-06 & 1.12E-05 & 5.73E-00 & 9.98E-00 & 557.1 & 6079.5 & 4498.3 \\
			Griewank ($n =1000$) & 9.33E-11 & 1.52E-10 & 4.47E-00 & 7.25E-00 & 643.5 & 203.6 & 6644.6 \\
			Rastrigin ($n =1000$) & 6.77E-07 & 7.83E-07 & 4.14E+03 & 9.03E+03 & 528.5 & 6179.8 & 5673.5 \\
			Schwefel ($n =1000$) & 1.27E-02 & 1.27E-02 & 1.57E+05 & 1.99E+05 & 739.5 & 10898.0 & 10826.0 \\
			Sphere ($n =1000$) & 7.7E-09 & 8.20E-09 & 3.48E+03 & 3.00E+03 & 238.3 & 321.0 & 6172.5 \\
			Sum squares ($n =1000$) & 3.8E-06 & 4.12E-06 & 1.20E+06 & 1.21E+02 & 280.7 & 534.1 & 21576.0 \\ \hline
	\end{tabular}}
	\caption{Comparative study of RMPS, GA and SA for minimizing 100 and 1000 dimensional Ackley, Griewank, Rastrigin, Schwefel, Sphere and Sum squares functions where the true solution is a boundary point. Benchmark functions are minimized starting from 10 randomly generated starting points from the corresponding domains using RMPS, GA and SA. The minimum of the 10 obtained minimum values of the objective functions is noted down for each method. For RMPS, also the maximum of these 10 obtained minimum objective function values is noted down as well. Average computation times over 10 starting point scenarios for each method is noted down (in seconds). RMPS outperforms other methods and the closeness of the min. and max. of the 10 obtained function minimum values (by RMPS, corresponding to 10 different starting points) implies the performance of RMPS is less dependent on the starting point, in general.}
	\label{table_high_dim_boundary}
\end{table}

\section*{Appendix D : Optimization on Hypervolume under manifolds}
\paragraph{Upper and Lower Bound Approach (ULBA) :}
\cite{Hsu2016} showed the following equality holds true
$$\max\{0, (H-1)P_{A}(\boldsymbol{\beta}) - (H-2)\} \le D(\boldsymbol{\beta}) \le P_{H}(\boldsymbol{\beta}),$$
where $P_{A}(\boldsymbol{\beta})$ and $P_{H}(\boldsymbol{\beta})$ are defined by
$$P_{A}(\boldsymbol{\beta}) =\frac{1}{H-1} \displaystyle\sum_{j=1}^{H-1}P(\boldsymbol{\beta}^{T}\mathbf{X}_{j+1}>\boldsymbol{\beta}^{T}\mathbf{X}_{j}), P_{H}(\boldsymbol{\beta}) = \min_{1 \le j \le H-1} P(\boldsymbol{\beta}^{T}\mathbf{X}_{j+1}>\boldsymbol{\beta}^{T}\mathbf{X}_{j}).$$ They proposed to maximize $P_{A}(\boldsymbol{\beta})$ in order to obtain the optimal biomarker combination. Compared to EHUM, ULBA objective function is much less expensive to compute. However, rest of the challenges for maximizing EHUM (e.g., discontinuity, possible multi-modal nature) also applies here as well.

\paragraph{Step-down algorithm :}
The step-down algorithm for maximizing any given objective function is given as follows:
\begin{description}
 \item[{\bf Step 1.}] EHUM values of the individual biomarkers are computed and based on their individual EHUM values, they are arranged in decreasing order. Suppose $X_{(1)}$ and $X_{(d)}$ denote the biomarkers with the highest and the lowest individual EHUM values, respectively.
 \item[{\bf Step 2.}] The first two biomarkers with the highest EHUM values are taken and combined as $V_2=X_{(1)} + \lambda_2 X_{(2)}$, where $\lambda_2$ is a parameter that needs to be estimated.
\item[{\bf Step 3.}] The objective function for the combined biomarker $V_2$ is maximized with respect to $\lambda_2$. Let $\widehat{V}_2 = X_{(1)} + \widehat{\lambda}_2 X_{(2)}$ denote the updated combination vector.
 \item[{\bf Step 4.}] For $i = 3,\ldots,d$ define $V_i = \widehat{V}_{i-1}+\lambda_i X_{(i)}$ and maximize the objective function with respect to $\lambda_i$.  The  combination vector obtained at $i$-th step is given by $\widehat{\lambda}_i$.
\end{description}
The estimated optimal marker using step-down algorithm is given by  $\widehat{V}_d = X_{(1)} + \widehat{\lambda}_2 X_{(2)} + \cdots + \widehat{\lambda}_d X_{(d)}$.

\section*{Appendix E : On parallel computation with RMPS : Matrix Completion Problem with SCAD penalty}
RMPS is parallelizable and up to $2n$ parallel threads can be used while solving a $n$-dimensional black-box problem. However in the simulation study part, the time required for optimizing each function is noted down for single thread computing only. To perform parallel computing in MATLAB, in case \texttt{parfor} loop is used instead of \texttt{for} loop, depending on the operations performed within the loops (e.g., objective function evaluation), a scenario might arise where \texttt{for} loop works faster than \texttt{parfor} loop. Because at the beginning of the \texttt{parfor} loop, an amount of time is spent in distributing the parallelizable works to different workers and after termination of parallel jobs, some amount of time is spent for gathering the results. However no such additional time is spent when job is done in single thread (via \texttt{for} loop). So, in case the objective function is not computationally expensive enough compared to the amount of time required for distribution and collection of results in parallel computing (using \texttt{parfor} loop), it is possible for the multi-threaded job to be more time consuming compared to the single-thread computation. We note that for all the considered simulation experiments, time required using single thread computation is faster than using parallel computation. Therefore, all the results regarding computation times are noted down for single threaded computing only. However it might not be the case if the objective function is actually computationally very expensive. In order to show the benefit of using parallel computation with RMPS we consider the following case study.

The problem of recovering an unknown matrix from only a given fraction of its entries is known as matrix completion problem. \cite{Candes2009} first proposed a method to recover a matrix from a few given entries solving a convex optimization problem. Later to solve this problem, \cite{Candes2010} minimized nuclear norm of the matrix subject to the constraint that the given entries of the matrix should be the same. In other words, suppose we have a matrix $\mathbf{Y}=(y_{ij})_{n \times n}$ with some missing values. In that case, as mentioned in \cite{Candes2010}, the complete matrix $\mathbf{Y}$ can be obtained by solving the following problem,
\begin{align*}
	&\text{minimize :} \; ||\mathbf{X}||_* \\
	&\text{subject to :} \; x_{ij}=y_{ij} \;\text{for all observed} \; (i,j),
\end{align*}
where $||\mathbf{M}||_*=\sum_i\sigma_i(\mathbf{M})$ denotes the nuclear norm, $\sigma_i(\mathbf{M})$ being the $i$-th singular value of matrix $\mathbf{M}$. This problem can be solved using convex optimization technique. On a closer look, it can be noted that minimizing nuclear norm in this fashion in similar to the LASSO (\cite{Tibshirani1996}) penalty term. \cite{Fan2001} proposed Smoothly Clipped Absolute Deviation (SCAD) penalty which was shown to have more desirable properties compared to LASSO for solving shrinkage based variable selection problems. But unlike LASSO, SCAD penalty is not a convex minimization (or concave maximization) problem. In this section, the matrix completion problem is solved using the SCAD penalty with RMPS. The matrix completion problem using SCAD penalty can be re-formulated as
\begin{align}
	&\text{minimize :} \; \sum_i f(\sigma_i(\mathbf{X})) \nonumber\\
	&\text{subject to :} \; x_{ij}=y_{ij} \;\text{for all observed} \; (i,j),
	\label{scad}
\end{align}
where $\sigma_i(\mathbf{X})$ are singular values and $f_i$ is the SCAD penalty function dependent of tuning parameters $\lambda$ and $a\;(=3.7)$ (see \cite{Fan2001}).

We consider a picture (Figure \ref{fig:GIVEN_PIC}) with $61\times 61$ pixels where approximately half (1877 to be precise) of its pixels are missing. The problem given by \eqref{scad} can been seen as a black-box function of dimension of 1877 (i.e., the number of missing pixels). It is also known that the numerical value of grey level of each pixel must be between 0 and 255. This problem is solved using RMPS method. We fit the model for 30 values of $\lambda$ which are $\{100,200, \ldots, 3000\}$ and only the best visual output ($\lambda=900$) is reported in Figure \ref{fig:SCAD_FINAL_POINT_900}.

\begin{figure}
	\begin{minipage}[b]{0.5\linewidth}
		\centering
		\includegraphics[width=0.6\linewidth]{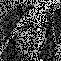} 
		\caption{Original picture containing missing pixels.} 
		\label{fig:GIVEN_PIC} 
	\end{minipage}
	\begin{minipage}[b]{0.5\linewidth}
		\centering
		\includegraphics[width=0.6\linewidth]{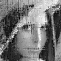} 
		\caption{Completed picture using SCAD penalized matrix completion with RMPS.} 
		\label{fig:SCAD_FINAL_POINT_900} 
	\end{minipage}
	\label{figfig6a} 
\end{figure}

Unlike the objective functions considered in the performance evaluation studies in the main paper, the evaluation of SCAD penalty based on the singular values of the matrix is very computationally intense. Thus, unlike previous cases, here using parallel computing is noted to be beneficial. It should be noted that this is a 1877 dimensional problem and therefore up to 3754 parallel threads can be used while solving it using RMPS algorithm. We use 4 parallel threads to derive the complete image given in Figure \ref{fig:SCAD_FINAL_POINT_900}. For comparison of computation time required by single threading and parallel threading with 4 threads, the required computation times for first 50, 100 and 200 iterations are provided for all cases in Table \ref{parallel}. We get more than 3 folds improvement in computation time using parallel threading (with 4 threads) instead of single threading.

\begin{table}[h]
	\centering
	\begin{tabular}{@{}cccc@{}}\hline
		Iterations & \begin{tabular}[c]{@{}c@{}}1 thread\\ (time)\end{tabular} & \begin{tabular}[c]{@{}c@{}}4 threads\\ (time)\end{tabular} & \begin{tabular}[c]{@{}c@{}}Improvement\\ (folds)\end{tabular} \\ \hline
		50 & 1325.93 & 391.38 & 3.39 \\
		100 & 3189.97 & 881.75 & 3.62 \\
		200 & 7622.24 & 2162.56 & 3.52 \\ \hline
	\end{tabular}
	\caption{Computation times (in seconds) required for first 50, 100 and 200 iterations of RMPS while solving matrix completion problem with SCAD penalty using single thread and 4 parallel threads.}
	\label{parallel}
\end{table}

\end{document}